\documentclass[11pt]{amsart}
\usepackage{amsmath,amssymb,amsthm}
\usepackage[margin=1.05in]{geometry}
\usepackage{booktabs}
\usepackage{array}
\usepackage{xcolor}
\usepackage{tikz}
\usepackage{graphicx}
\usepackage[colorlinks=true,linkcolor=blue!50!black,citecolor=green!40!black,urlcolor=blue!50!black]{hyperref}

\newtheorem{theorem}{Theorem}[section]
\newtheorem{proposition}[theorem]{Proposition}
\newtheorem{lemma}[theorem]{Lemma}
\newtheorem{corollary}[theorem]{Corollary}
\newtheorem{conjecture}[theorem]{Conjecture}
\theoremstyle{definition}

\theoremstyle{remark}

\newtheorem{exa}{Example}[section]

\newcommand{\N}{\mathbb{N}}
\newcommand{\Z}{\mathbb{Z}}
\newcommand{\R}{\mathbb{R}}
\newcommand{\Imm}{\operatorname{Imm}}
\newcommand{\TNN}{\mathrm{TNN}}

\newcommand{\Ind}{\operatorname{Ind}}
\newcommand{\per}{\operatorname{per}}
\newcommand{\CF}{\mathrm{CF}}
\newcommand{\CKL}{C_{\mathrm{KL}}}
\newcommand{\CTNN}{C_{\mathrm{TNN}}}
\newcommand{\CJ}{C_{J}}
\newcommand{\Cp}[1]{C'_{#1}(1)}
\DeclareMathOperator{\ch}{ch}
\newcommand{\supp}{\operatorname{supp}}

\title[Haiman's conjecture for $n\le10$]{Haiman's conjecture for $n\le10$:\\ Kazhdan--Lusztig positivity of monomial immanants}
\author{Keith McNulty}
\email{keith.mcnulty@gmail.com}
\date{2 November 2026}

\makeatletter\g@addto@macro\UrlBreaks{\do\_\do\-}\makeatother
\begin{document}

\begin{abstract}
Haiman conjectured that for every partition $\lambda\vdash n$ and every $w\in S_n$ the polynomial $c_{\lambda,w}(q)$, a monomial character of the Hecke algebra evaluated at the Kazhdan--Lusztig basis, has non-negative integer coefficients and is symmetric and unimodal, and he verified this by computer for $n\le7$. We verify it for every $n\le10$. A stability lemma reduces the statement to $n=10$, which we settle by an exact computation that stores only the $1{,}252{,}173{,}030$ pairs $(x,w)$ in $S_{10}$ with $x$ maximal in its double coset for the descent sets of $w$, and takes the traces of the standard basis from the class polynomials of Geck and Pfeiffer. Every stored Kazhdan--Lusztig polynomial agrees with du~Cloux's \textsf{Coxeter3}, and the coefficients pass several independent checks, including a comparison at all $16{,}796$ codominant $w$ with chromatic quasisymmetric functions computed from colourings.

Consequently, for $n\le10$ every monomial immanant is a non-negative integer combination of Kazhdan--Lusztig immanants, Stembridge's monomial immanant conjecture holds, and a class-function immanant is non-negative on totally non-negative matrices if and only if it is non-negative at the block matrices $J_\mu$. We use this to describe the order of normalised irreducible-character immanants on totally non-negative matrices for $n\le10$, and to prove Pate's inequalities there. On the full span of permutation monomials we give a Newton-polytope criterion for the coefficient of $\Imm_w$ to be non-negative throughout the totally non-negative cone, show that for $n=4,5$ this happens exactly when $X_w$ is smooth, and check by exact computation for $n\le10$ the proof of this for all $n$ given in a companion paper.
\end{abstract}

\maketitle

\section{Introduction}\label{sec:intro}

A real matrix is \emph{totally non-negative} (TNN) if all its minors are non-negative; $\TNN_n$ denotes the set of such $n\times n$ matrices. For $f\colon S_n\to\R$ and an $n\times n$ matrix $x$ put $d_f(x)=\sum_{\sigma}f(\sigma)\,x_{1,\sigma(1)}\cdots x_{n,\sigma(n)}$. For an irreducible character $\chi^\lambda$ this is the immanant $\Imm_{\chi^\lambda}$, and Stembridge \cite{Ste91} proved that it is non-negative on $\TNN_n$. The monomial immanants $\Imm_\lambda=d_{\phi^\lambda}$, where $\phi^\lambda$ is the monomial virtual character, refine the irreducible ones, and Stembridge conjectured \cite[Conj.~2.1]{Ste92} that they are non-negative on $\TNN_n$ too. Haiman \cite{Hai} proposed a stronger statement in the Hecke algebra. Let $H_n$ be the Hecke algebra of $S_n$ over $\Z[q^{\pm1/2}]$, with $(T_s-q)(T_s+1)=0$, let $P_{v,w}(q)$ be the Kazhdan--Lusztig polynomials, so that $q^{\ell(w)/2}C'_w=\sum_{v\le w}P_{v,w}(q)\,T_v$, and let $\phi^\lambda_q$ be the monomial characters of $H_n$ (Section~\ref{sec:hecke}).

\begin{conjecture}[Haiman {\cite[Conj.~2.1]{Hai}}]\label{conj:Hq}
For every $\lambda\vdash n$ and $w\in S_n$, the polynomial $c_{\lambda,w}(q)=\phi^\lambda_q(q^{\ell(w)/2}C'_w)$ has non-negative integer coefficients, and it is symmetric and unimodal about $q^{\ell(w)/2}$.
\end{conjecture}

At $q=1$, $C'_w$ specialises to $\Cp w=\sum_{v\le w}P_{v,w}(1)\,v$, and Conjecture~\ref{conj:Hq} gives the following.

\begin{conjecture}[Haiman, Stembridge; $q=1$]\label{conj:H}
For every $\lambda\vdash n$ and $w\in S_n$, $c_{\lambda,w}=c_{\lambda,w}(1)=\phi^\lambda(\Cp w)\ge 0$.
\end{conjecture}

Rhoades and Skandera \cite{RS06} showed that $\Imm_\lambda=\sum_wc_{\lambda,w}\Imm_w$, where the Kazhdan--Lusztig immanants $\Imm_w$ are non-negative on $\TNN_n$. So Conjecture~\ref{conj:H} implies Stembridge's conjecture. It is known for shapes with at most two columns \cite{CSS,RS06}, and for smooth $w$ and every shape \cite{CHSS,AN,Hik}, \cite[Prop.~25]{Ska25}; the smooth case depends on Hikita's theorem \cite{Hik}. Haiman verified Conjecture~\ref{conj:Hq} by computer for $n\le7$ \cite[\S4]{Hai}; earlier, Stembridge had checked his conjecture for $n\le4$ and for $n=5$ apart from two shapes \cite{Ste92}. The companion preprint \cite{McN} gives a proof for all hooks and every $n$. Non-negativity of hook monomial immanants on $\TNN_n$, Stembridge's conjecture for hooks, was proved earlier by Dong, Wu and Gao \cite{DWG}, who need, for the hook $(n-k,1^k)$, only the minors of order at most $k+1$ to be non-negative. Our main result extends Haiman's verification to $n\le10$.

\begin{theorem}\label{thm:main}
For every $n\le10$, every $\lambda\vdash n$ and every $w\in S_n$, $c_{\lambda,w}(q)$ has non-negative integer coefficients and is symmetric and unimodal about $q^{\ell(w)/2}$. In particular $c_{\lambda,w}=c_{\lambda,w}(1)$ is a non-negative integer.
\end{theorem}

For codominant ($312$-avoiding) $w$, $c_{\lambda,w}(q)$ is the coefficient of $e_\lambda$ in the chromatic quasisymmetric function of the indifference graph of $w$ (see \cite{CHSS,SW16}; for $n\le10$ we check this identity directly, Section~\ref{sec:checks}, check~13), and Conjecture~\ref{conj:Hq} at these $w$ is the conjecture of Shareshian and Wachs that these functions are $e$-positive and $e$-unimodal \cite[Conj.~5.1]{SW16}. At $q=1$, $e$-positivity is the Stanley--Stembridge conjecture for natural unit interval orders, proved by Hikita \cite{Hik}; the $e$-positivity of the individual coefficients of $q^j$, and $e$-unimodality, remain open in general. Shareshian and Wachs checked their conjecture by computer for the graphs $G_{n,r}$ with $n\le8$ \cite[\S5]{SW16}, and through \cite{CHSS} Haiman's computation \cite[\S4]{Hai} covers all natural unit interval orders with at most $7$ elements. Theorem~\ref{thm:main} covers all those with at most $10$ elements; we are not aware of a published check of all natural unit interval orders with $8$ or more elements.

\subsection*{Three cones}
We state the consequences in terms of three cones in the space $\CF_n$ of real class functions on $S_n$:
\begin{align*}
\CKL&=\{\phi:\ \phi(\Cp w)\ge0\text{ for all }w\in S_n\},\\
\CTNN&=\{\phi:\ d_\phi\ge0\text{ on }\TNN_n\},\\
\CJ&=\{\phi:\ d_\phi(J_\mu)\ge0\text{ for all }\mu\vdash n\},
\end{align*}
where $J_\mu$ is the block-diagonal matrix with all-ones blocks of sizes $\mu_1,\mu_2,\dots$. For every $n$, $\CKL\subseteq\CTNN\subseteq\CJ$, and $\CJ$ is the simplicial cone spanned by the monomial characters (Section~\ref{sec:cones}). So Stembridge's conjecture is the equality $\CTNN=\CJ$, and Conjecture~\ref{conj:H} is $\CKL=\CJ$. Theorem~\ref{thm:main} gives both for $n\le10$.

\begin{corollary}\label{cor:cones}
For $n\le10$, $\CKL=\CTNN=\CJ$. Hence every monomial immanant of size $n\le10$ is non-negative on $\TNN_n$, and a class-function immanant $d_\phi$ is non-negative on $\TNN_n$ iff $d_\phi(J_\mu)\ge0$ for all $\mu\vdash n$, iff it is a non-negative combination of Kazhdan--Lusztig immanants.
\end{corollary}

For the normalised immanants $N_a=\Imm_{\chi^a}/\chi^a(e)$, Corollary~\ref{cor:cones} turns every inequality $N_a\ge N_b$ on $\TNN_n$ into a finite check on Kostka numbers $K_{a,\mu}$.

\begin{theorem}\label{thm:order}
Let $n\le10$ and $a,b\vdash n$, and write $f^a=\chi^a(e)$. Then $N_a\ge N_b$ on $\TNN_n$ iff $K_{a,\mu}/f^a\ge K_{b,\mu}/f^b$ for every $\mu\vdash n$. If the criterion fails at $\mu$, then $N_a(J_\mu)<N_b(J_\mu)$.
\end{theorem}

For $a,b\vdash n$ write $a\succeq b$ if $N_a\ge N_b$ on $\TNN_n$. This is a partial order on the partitions of $n$ for every $n$, and it is contained in dominance order, since $N_a(J_b)\ge N_b(J_b)>0$ forces $K_{a,b}>0$ (Section~\ref{sec:consequences}). For $n\le10$ it equals dominance only when $n\le3$, and it is a lattice only when $n\le5$ (Theorem~\ref{thm:structure}); we make no claim about its structure for $n\ge11$. We also obtain Pate's inequalities on $\TNN_n$ for $n\le10$ (Corollary~\ref{cor:pate}), answering Skandera's question \cite[Problem~29]{Ska25} in this range, and an inequality between several normalised immanants that is not a consequence of two-term inequalities (Proposition~\ref{prop:notpairwise}).

\subsection*{The computation}
Lemma~\ref{lem:stability} shows that $c_{\lambda,w}$ is unchanged when $w$ is extended by fixed points and $\lambda$ by parts equal to $1$, so Theorem~\ref{thm:main} for $n=10$ implies it for all $n\le10$. A direct computation of all Kazhdan--Lusztig polynomials of $S_{10}$ is out of reach: $S_9$ already has $10{,}501{,}351{,}657$ Bruhat pairs, and $S_{10}$ about $10^{12}$. By descent invariance it suffices to store $P_{x,w}$ for the $x$ that are maximal in their double coset for the descent sets of $w$. In $S_{10}$ there are $1{,}252{,}173{,}030$ such \emph{extremal pairs}. Working with extremal pairs is not new: du~Cloux's \textsf{Coxeter3} \cite{Cox3} uses them, and Warrington's tables of Kazhdan--Lusztig polynomials store a compressed set of them for $S_4,\dots,S_9$ \cite{War}. Our contributions are their enumeration as contingency tables (Lemma~\ref{lem:corner}) and the scale of the computation. We run the Kazhdan--Lusztig recursion on the extremal pairs (Section~\ref{sec:recursion}) and sum characters over double cosets (Section~\ref{sec:coeff}); for the graded statement the characters are replaced by the traces $\phi^\lambda_q(T_v)$, computed from the class polynomials of Geck and Pfeiffer \cite{GP} (Section~\ref{sec:graded}). At $q=1$ the computation takes under five minutes on $192$ cores, and the graded one an hour and a half. Section~\ref{sec:checks} lists the checks, including an exhaustive comparison of all $1.25\times10^9$ extremal pairs with du~Cloux's \textsf{Coxeter3} \cite{Cox3}.

\subsection*{Beyond class functions}
On the span $\Pi_n$ of all permutation monomials, the analogous cones $\CKL(\Pi_n)\subseteq\CTNN(\Pi_n)$ differ for $n\ge4$ \cite{RS06,Ska08}. Answering a question of Rhoades and Skandera \cite{RS06}, Skandera \cite{Ska08} showed that whenever the Schubert variety $X_w$ is smooth, the coefficient of $\Imm_w$ is non-negative in every element of $\CTNN(\Pi_n)$. We give a vertex criterion on Newton polytopes for this property (Theorem~\ref{thm:isolation}), and show with it that for $n\le5$ it holds exactly when $X_w$ is smooth (Theorem~\ref{thm:smalln}). The companion paper gives a proof of this for every $n$ \cite[Thm.~9.6]{McN}, by a different method, and Section~\ref{sec:evidence} checks that proof for $n\le10$. So the restriction to class functions in Corollary~\ref{cor:cones} is essential.

The companion paper \cite{McN} gives proofs of the results that we state for every $n$, among them Kazhdan--Lusztig positivity of all hook monomial immanants. Section~\ref{sec:evidence} records the computations that support the conjectures stated there. The paper is organised linearly: Sections~\ref{sec:prelim}--\ref{sec:stability} set up the cones and the reduction to $n=10$, Sections~\ref{sec:extremal}--\ref{sec:checks} carry out the computation, Section~\ref{sec:consequences} derives the consequences, Section~\ref{sec:beyond} treats the full span $\Pi_n$, and Section~\ref{sec:evidence} collects further computations. The code, the verification scripts and the result files are archived as \cite{McNcode}, and the complete table of coefficients for $n=10$ separately as \cite{McNdata}; Appendix~\ref{app:data} explains how to reproduce and check the computation.

\section{Preliminaries}\label{sec:prelim}

\subsection{Class functions}
For a composition $\alpha$ of $n$ let $S_\alpha=S_{\alpha_1}\times S_{\alpha_2}\times\cdots$ be the Young subgroup preserving the consecutive blocks of sizes $\alpha_1,\alpha_2,\dots$, let $\alpha!=|S_\alpha|$, and let $w_0(\alpha)$ be the longest element of $S_\alpha$. For $\mu\vdash n$ put $\eta^\mu=\Ind_{S_\mu}^{S_n}1$. With $\langle\phi,\psi\rangle=\frac1{n!}\sum_\sigma\phi(\sigma)\psi(\sigma)$, the monomial characters $\phi^\lambda$ are the dual basis: $\langle\phi^\lambda,\eta^\mu\rangle=\delta_{\lambda\mu}$. Under the Frobenius characteristic, $\eta^\mu\mapsto h_\mu$, $\chi^\lambda\mapsto s_\lambda$ and $\phi^\lambda\mapsto m_\lambda$. By Young's rule $\langle\chi^a,\eta^\mu\rangle=K_{a,\mu}$. This is the dimension of the space of $S_\mu$-invariants in the irreducible module $V^a$, and $K_{a,\mu}\ne0$ iff $a$ dominates $\mu$.

\subsection{Kazhdan--Lusztig immanants}
Rhoades and Skandera \cite{RS06} define
\[
  \Imm_w(x)=\sum_{v\ge w}(-1)^{\ell(v)-\ell(w)}P_{w_0v,w_0w}(1)\,x_{1,v(1)}\cdots x_{n,v(n)},
\]
and prove that $\Imm_w\ge0$ on $\TNN_n$ and that, for every $f\colon S_n\to\R$,
\begin{equation}\label{eq:duality}
  d_f=\sum_{w\in S_n}f(\Cp w)\,\Imm_w .
\end{equation}
We call $d_f$ \emph{KL-positive} if $f(\Cp w)\ge0$ for all $w$. Since $P_{v,w}=P_{v^{-1},w^{-1}}=P_{w_0vw_0,w_0ww_0}$, we have $c_{\lambda,w}=c_{\lambda,w^{-1}}=c_{\lambda,w_0ww_0}$.

\subsection{The Hecke algebra}\label{sec:hecke}
$H_n$ has basis $\{T_v\}_{v\in S_n}$, with $T_vT_s=T_{vs}$ if $\ell(vs)>\ell(v)$ and $(T_s-q)(T_s+1)=0$. For a composition $\alpha$ let $H_\alpha$ be the parabolic subalgebra spanned by the $T_v$, $v\in S_\alpha$, and let $\eta^\mu_q$ be the character of $H_n$ induced from the index representation $T_s\mapsto q$ of $H_\mu$. The \emph{monomial characters} $\phi^\lambda_q$ are defined from the $\eta^\mu_q$ by the same inverse-Kostka relation as at $q=1$, so that $\chi^a_q=\sum_\lambda K_{a,\lambda}\phi^\lambda_q$ and the Frobenius characteristic of $\phi^\lambda_q$ is $m_\lambda$; at $q=1$ they specialise to $\phi^\lambda$. Put
\[
c_{\lambda,w}(q)=\phi^\lambda_q\bigl(q^{\ell(w)/2}C'_w\bigr)=\sum_{v\le w}P_{v,w}(q)\,\phi^\lambda_q(T_v),
\]
so that $c_{\lambda,w}(1)=c_{\lambda,w}$. Since $P_{v,w}=P_{v^{-1},w^{-1}}=P_{w_0vw_0,w_0ww_0}$ and $\tau(T_v)=\tau(T_{v^{-1}})=\tau(T_{w_0vw_0})$ for every trace $\tau$ of $H_n$, also $c_{\lambda,w}(q)=c_{\lambda,w^{-1}}(q)=c_{\lambda,w_0ww_0}(q)$. The value $\phi^\lambda_q(T_v)$ is not a function of the conjugacy class of $v$; Section~\ref{sec:graded} explains how it is computed.

\begin{exa}\label{exa:graded}
For $w=3412$ and $\lambda=(4),(3,1),(2,2),(2,1,1),(1^4)$,
\[
c_{\cdot,3412}(q)=\bigl(1+2q+2q^2+2q^3+q^4,\ q(1+q)^2,\ q(1+q)^2,\ 0,\ 0\bigr),
\]
each symmetric and unimodal about $q^{\ell(w)/2}=q^2$; at $q=1$ this is $(8,4,4,0,0)$ (Example~\ref{exa:coeff}).
\end{exa}

\subsection{Block reduction}
Say $w=w_1\times w_2$ if $w$ maps $\{1,\dots,m\}$ to itself, where $w_1\in S_m$ and $w_2\in S_{n-m}$ are the restrictions. Every $v\le w$ then lies in $S_m\times S_{n-m}$, and $P_{v,w}=P_{v_1,w_1}P_{v_2,w_2}$, so $\Cp{w}=\Cp{w_1}\otimes\Cp{w_2}$ in $\Z[S_m]\otimes\Z[S_{n-m}]$. Restricting $\phi^\lambda$ to $S_m\times S_{n-m}$ corresponds to the coproduct $\Delta m_\lambda=\sum_{\mu\cup\nu=\lambda}m_\mu\otimes m_\nu$, where $\mu\cup\nu$ is the partition with the parts of both. This gives:

\begin{lemma}\label{lem:coproduct}
$c_{\lambda,w_1\times w_2}=\sum_{\mu\cup\nu=\lambda}c_{\mu,w_1}\,c_{\nu,w_2}$.
\end{lemma}

\section{Three cones}\label{sec:cones}

\begin{lemma}\label{lem:J}
\begin{enumerate}
\item For every composition $\alpha$ of $n$, $J_\alpha\in\TNN_n$.
\item For every class function $\phi$, $d_\phi(J_\mu)=\mu!\,\langle\phi,\eta^\mu\rangle$. In particular $\Imm_\lambda(J_\mu)=\mu!\,\delta_{\lambda\mu}$ and $\Imm_{\chi^a}(J_\mu)=\mu!\,K_{a,\mu}$.
\item For every $w\in S_n$, $\Imm_w(J_\alpha)=\delta_{w,w_0(\alpha)}$. In particular $c_{\lambda,w_0(\alpha)}=\alpha!\,\delta_{\lambda,\mathrm{sort}(\alpha)}$.
\end{enumerate}
\end{lemma}

\begin{proof}
(1) Order the rows and columns of a square submatrix by the blocks they lie in. If some block contributes more rows than columns, those rows are supported on too few columns and the minor vanishes. Otherwise the minor is a product of determinants of all-ones matrices, each equal to $0$ or $1$.
(2) $\prod_i(J_\mu)_{i,\sigma(i)}=[\sigma\in S_\mu]$, so $d_\phi(J_\mu)=\sum_{\sigma\in S_\mu}\phi(\sigma)=\mu!\,\langle\phi,\eta^\mu\rangle$ by Frobenius reciprocity.
(3) The lower interval of $w_0(\alpha)$ is $S_\alpha$, and $P_{v,w_0(\alpha)}=1$ for $v\in S_\alpha$, so $\Cp{w_0(\alpha)}=\sum_{v\in S_\alpha}v$. For every $f$, \eqref{eq:duality} at $J_\alpha$ gives $\sum_wf(\Cp w)\Imm_w(J_\alpha)=\sum_{v\in S_\alpha}f(v)=f(\Cp{w_0(\alpha)})$. Since $\{\Cp w\}$ is a basis and $f$ is arbitrary, $\Imm_w(J_\alpha)=\delta_{w,w_0(\alpha)}$. The last statement follows from (2).
\end{proof}

\begin{theorem}\label{thm:cones}
For every $n$, $\CKL\subseteq\CTNN\subseteq\CJ$, and $\CJ=\operatorname{cone}\{\phi^\lambda:\lambda\vdash n\}$ is simplicial. Hence Stembridge's conjecture for $n$ is equivalent to $\CTNN=\CJ$, and Conjecture~\ref{conj:H} for $n$ is equivalent to $\CKL=\CJ$.
\end{theorem}

\begin{proof}
The first inclusion is \eqref{eq:duality} together with $\Imm_w\ge0$ on $\TNN_n$, and the second is Lemma~\ref{lem:J}(1). By Lemma~\ref{lem:J}(2), $\phi\in\CJ$ iff $\langle\phi,\eta^\mu\rangle\ge0$ for all $\mu$, iff $\phi$ has non-negative coordinates in the basis $\{\phi^\lambda\}$ dual to $\{\eta^\mu\}$. So $\CTNN=\CJ$ iff every generator $\phi^\lambda$ lies in $\CTNN$, and similarly for $\CKL$.
\end{proof}

The inclusion $\CKL\subseteq\CJ$ also follows directly from Lemma~\ref{lem:J}(3), by evaluation at the elements $w_0(\mu)$. This is due to Haiman, who observed that any class function with the positivity of Conjecture~\ref{conj:H} is a non-negative combination of monomial characters \cite[\S2 and Prop.~4.1(1)]{Hai}. Theorem~\ref{thm:cones} uses no computation and holds for every $n$. The computation enters the three-cone picture only through the equality $\CKL=\CJ$ for $n\le10$ (Corollary~\ref{cor:cones}).

\section{Stability}\label{sec:stability}

\begin{lemma}[stability]\label{lem:stability}
Let $\lambda\vdash m$, $w\in S_m$, $k\ge0$, and let $e_k$ be the identity of $S_k$. Then $c_{\lambda\cup1^k,\;w\times e_k}=c_{\lambda,w}$.
\end{lemma}

\begin{proof}
By Lemma~\ref{lem:coproduct}, $c_{\lambda\cup1^k,w\times e_k}=\sum_{\mu\cup\nu=\lambda\cup1^k}c_{\mu,w}\,c_{\nu,e_k}$, with $\nu\vdash k$. Now $\Cp{e_k}=e$, and $c_{\nu,e_k}=\phi^\nu(e)=\langle\phi^\nu,\eta^{1^k}\rangle=\delta_{\nu,1^k}$, because $\eta^{1^k}$ is the regular character. The only term left is $\nu=1^k$, $\mu=\lambda$.
\end{proof}

\begin{corollary}\label{cor:down}
If Conjecture~\ref{conj:H} holds for $N$, it holds for every $n\le N$. The same holds for Stembridge's conjecture.
\end{corollary}

\begin{proof}
For $\lambda\vdash n$ and $w\in S_n$, Lemma~\ref{lem:stability} gives $c_{\lambda,w}=c_{\lambda\cup1^{N-n},\,w\times e_{N-n}}\ge0$. For Stembridge's conjecture, let $A\in\TNN_n$. Only $\sigma\in S_n\times\{e\}$ contribute to $d_\phi(A\oplus I_{N-n})$, so $d_\phi(A\oplus I_{N-n})=d_{\psi}(A)$ with $\psi(\sigma)=\phi(\sigma\times e)$. For $\phi=\phi^{\lambda\cup1^{N-n}}$ the coproduct and $\phi^\nu(e)=\delta_{\nu,1^{N-n}}$ give $\psi=\phi^\lambda$, so $\Imm_\lambda(A)=\Imm_{\lambda\cup1^{N-n}}(A\oplus I)\ge0$.
\end{proof}

Both statements hold at $q$.

\begin{lemma}[graded coproduct and stability]\label{lem:qstability}
For $w=w_1\times w_2$ with $w_1\in S_m$, $c_{\lambda,w}(q)=\sum_{\mu\cup\nu=\lambda}c_{\mu,w_1}(q)\,c_{\nu,w_2}(q)$. In particular $c_{\lambda\cup1^k,\,w\times e_k}(q)=c_{\lambda,w}(q)$, and if Conjecture~\ref{conj:Hq} holds for $N$ then it holds for every $n\le N$.
\end{lemma}

\begin{proof}
Since $\ell(w)=\ell(w_1)+\ell(w_2)$ and $P_{v,w}=P_{v_1,w_1}P_{v_2,w_2}$, $q^{\ell(w)/2}C'_w$ is the tensor product of $q^{\ell(w_i)/2}C'_{w_i}$ in $H_m\otimes H_{n-m}=H_{(m,n-m)}\subseteq H_n$. The restriction of $\eta^\kappa_q$ to this parabolic subalgebra is given by the Mackey formula for Hecke algebras \cite[\S9.1]{GP}. Explicitly, the double cosets $S_{(m,n-m)}\backslash S_n/S_\kappa$ correspond to the $2\times\ell(\kappa)$ contingency tables with row sums $(m,n-m)$ and column sums $\kappa$, that is, to the pairs of weak compositions $\rho$ of $m$ and $\sigma$ of $n-m$ with $\rho_i+\sigma_i=\kappa_i$. For the table $(\rho,\sigma)$, the intersection of $S_{(m,n-m)}$ with the conjugate of $S_\kappa$ by the distinguished double coset representative is $S_\rho\times S_\sigma$, and the index representation restricts to the index representation. So
\[
\operatorname{Res}^{H_n}_{H_m\otimes H_{n-m}}\eta^\kappa_q=\sum_{\rho+\sigma=\kappa}\eta^\rho_q\otimes\eta^\sigma_q ,
\]
where $\eta^\rho_q$ for a weak composition means $\eta^{\mathrm{sort}(\rho)}_q$. This is the same expansion as $\Delta h_\kappa=\sum_{\rho+\sigma=\kappa}h_\rho\otimes h_\sigma$. Now $\phi^\lambda_q$ is the same integer combination of the $\eta^\kappa_q$ as $m_\lambda$ is of the $h_\kappa$, so $\operatorname{Res}\phi^\lambda_q$ is the combination of the $\eta^\rho_q\otimes\eta^\sigma_q$ with the coefficients of $\Delta m_\lambda$ in the basis $\{h_\rho\otimes h_\sigma\}$. Since $\Delta m_\lambda=\sum_{\mu\cup\nu=\lambda}m_\mu\otimes m_\nu$, this gives $\operatorname{Res}\phi^\lambda_q=\sum_{\mu\cup\nu=\lambda}\phi^\mu_q\otimes\phi^\nu_q$ and the first formula. For the second use $\phi^\nu_q(T_e)=\phi^\nu(e)=\delta_{\nu,1^k}$, as in Lemma~\ref{lem:stability}; symmetry and unimodality are preserved because $\ell(w\times e_k)=\ell(w)$.
\end{proof}

So Theorem~\ref{thm:main} follows from the case $n=10$, which is Theorem~\ref{thm:n10} below.

\begin{exa}\label{exa:stability}
For $w=231\in S_3$, adding a fixed point gives $w\times e_1=2314\in S_4$, and
\[
c_{(3),231}=c_{(3,1),2314}=3,\qquad c_{(2,1),231}=c_{(2,1,1),2314}=1,\qquad c_{(1^3),231}=c_{(1^4),2314}=0 .
\]
The remaining shape $(2,2)\vdash4$ is not of the form $\lambda\cup1$, and indeed $c_{(2,2),2314}=0$, as Lemma~\ref{lem:coproduct} requires: no $\mu\vdash3$ has $\mu\cup(1)=(2,2)$, so the sum is empty.
\end{exa}

\section{Extremal pairs}\label{sec:extremal}

\begin{lemma}[{\cite[(2.3.g)]{KL}}]\label{lem:descent}
If $s\in D_L(w)$ then $P_{x,w}=P_{sx,w}$ for all $x$. If $t\in D_R(w)$ then $P_{x,w}=P_{xt,w}$ for all $x$.
\end{lemma}

Let $I=D_L(w)$ and $J=D_R(w)$, and let $W_I,W_J$ be the corresponding Young subgroups. By Lemma~\ref{lem:descent}, $P_{x,w}$ is constant on each double coset $W_IxW_J$. The lower interval $\{x\le w\}$ is a union of such double cosets, by the lifting property. So it suffices to know $P_{x,w}$ for $x$ maximal in its double coset. These are exactly the $x\le w$ with $D_L(x)\supseteq I$ and $D_R(x)\supseteq J$, and we call $(x,w)$ an \emph{extremal pair}. Write $\hat y=\max(W_IyW_J)$, so that $P_{y,w}=P_{\hat y,w}$.

\begin{proposition}\label{prop:counts}
$S_9$ has $33{,}850{,}056$ extremal pairs and $S_{10}$ has $1{,}252{,}173{,}030$. There are $169{,}438$ distinct pairs $(D_L(w),D_R(w))$ in $S_{10}$.
\end{proposition}

\emph{Double cosets as contingency tables.} Let $\alpha$ be the composition of positions cut by $J$ and $\beta$ the composition of values cut by $I$. Right multiplication by $W_J$ permutes positions within the blocks of $\alpha$, and left multiplication by $W_I$ permutes values within the blocks of $\beta$. So the double coset of $y$ is determined by its table $T(y)_{ab}=\#\{i\in\alpha_a:\ y(i)\in\beta_b\}$, a non-negative integer matrix with row sums $\alpha$ and column sums $\beta$, and every such matrix occurs. The double coset with table $T$ has size $\prod_a\alpha_a!\prod_b\beta_b!/\prod_{a,b}T_{ab}!$. Its maximal element is obtained as follows: in each position block, the values appear in decreasing order; each position block takes, from each value block $b$, its $T_{ab}$ largest unused values; and earlier position blocks receive larger values.

\begin{lemma}\label{lem:corner}
Number the position blocks $a=1,\dots,p$ from the top and the value blocks $b=1,\dots,r$ in increasing order of values, and put $i_a=\alpha_1+\dots+\alpha_a$ and $j_b=\beta_1+\dots+\beta_{b-1}+1$, the last position of block $a$ and the smallest value of block $b$. Let $x$ and $w$ be maximal in their $(W_I,W_J)$-double cosets, with tables $T(x)$ and $T(w)$. Then $x\le w$ iff the corner sums $S_T(a,b)=\sum_{a'\le a,\ b'\ge b}T_{a'b'}$ satisfy $S_{T(x)}(a,b)\le S_{T(w)}(a,b)$ for all $1\le a<p$ and $1<b\le r$. (For $a=p$ or $b=1$ the corner sum is a sum of margins, the same for $x$ and $w$.)
\end{lemma}

\begin{proof}
By the tableau criterion, $x\le w$ iff $r_x(i,j)\le r_w(i,j)$ for all $i,j$, where $r_y(i,j)=\#\{i'\le i:\ y(i')\ge j\}$. At block corners, $r_y(i_a,j_b)=S_{T(y)}(a,b)$, which proves necessity.

For sufficiency, let $i_0<i\le i_1$ lie in a position block. For a maximal element the values in the block decrease, so the positions of the block with value $\ge j$ come first. Hence $r_y(i,j)=\min\bigl(r_y(i_0,j)+(i-i_0),\ r_y(i_1,j)\bigr)$. Similarly, let $j$ lie in a value block with values $j_0<\dots<j_1$, and let $i_0$ be a position-block boundary. The values of that block used by positions $\le i_0$ are its largest ones, so $r_y(i_0,j)=\min\bigl(r_y(i_0,j_1+1)+(j_1+1-j),\ r_y(i_0,j_0)\bigr)$, where $r_y(i_0,j_1+1)$ and $r_y(i_0,j_0)$ are values at block corners. These expressions are monotone in the corner values and identical for $x$ and $w$, so the corner inequalities propagate to all $(i,j)$ (Figure~\ref{fig:rank}).
\end{proof}

\begin{figure}[ht]
\centering
\begin{tikzpicture}[scale=0.75]
  \fill[blue!12] (2,0) rectangle (4,-2);
  \foreach \i in {0,...,4} { \draw[gray!50] (0,-\i) -- (4,-\i); \draw[gray!50] (\i,0) -- (\i,-4); }
  \draw[very thick] (0,-1) -- (4,-1); \draw[very thick] (0,-3) -- (4,-3);
  \draw[very thick] (1,0) -- (1,-4); \draw[very thick] (3,0) -- (3,-4);
  \draw[thick] (0,0) rectangle (4,-4);
  \foreach \i/\v in {1/3,2/4,3/1,4/2} { \fill (\v-0.5,-\i+0.5) circle (4pt); }
  \foreach \i in {1,...,4} { \node[left] at (0,-\i+0.5) {\small $\i$}; \node[above] at (\i-0.5,0) {\small $\i$}; }
  \node[left] at (-0.5,-2) {\small position $i$};
  \node[above] at (2,0.45) {\small value $j$};
  \node[right,align=left] at (4.3,-1) {\small $r_w(2,3)=2$};
  \node[right,align=left] at (4.3,-2.2) {\small $=\min\bigl(r_w(1,3)+1,\ r_w(3,3)\bigr)$};
  \node[right,align=left] at (4.3,-3.0) {\small $=\min(1+1,\,2)$};
\end{tikzpicture}
\caption{The permutation $w=3412$ (a dot in row $i$, column $w(i)$) with the block grid of Example~\ref{exa:tables} drawn thick. The shaded region counts $r_w(2,3)=\#\{i'\le2:\ w(i')\ge3\}$. Row $2$ is inside the position block $\{2,3\}$, and $r_w(2,3)$ is determined by the corner values $r_w(1,3)$ and $r_w(3,3)$, as in the proof of Lemma~\ref{lem:corner}.}\label{fig:rank}
\end{figure}
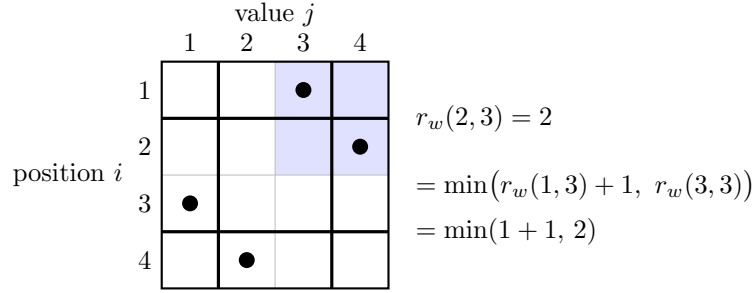

So the extremal $x\le w$ are enumerated directly, as the tables whose corner sums are dominated by those of $T(w)$, without listing the Bruhat interval.

\begin{exa}\label{exa:tables}
Let $w=3412$. Then $D_L(w)=D_R(w)=\{s_2\}$, so $\alpha=\beta=(1,2,1)$: the position blocks are $\{1\},\{2,3\},\{4\}$ and so are the value blocks. The double cosets of $S_4$ correspond to the $3\times3$ tables with row and column sums $(1,2,1)$, and there are seven of them. Their sizes are $4/\prod T_{ab}!$, which sum to $24$. The table lists each maximal element, its table, the size of its double coset, and the corner sums $S_T(a,b)$ for $a\in\{1,2\}$, $b\in\{2,3\}$; the other corner sums are fixed by the margins.
\begin{center}
\small
\renewcommand{\arraystretch}{1.9}
\begin{tabular}{cccccc}
\toprule
maximal element & $T$ & size & $\begin{pmatrix}S_T(1,2)&S_T(1,3)\\S_T(2,2)&S_T(2,3)\end{pmatrix}$ & $\le w$ \\
\midrule
$1324$ & $\left(\begin{smallmatrix}1&0&0\\0&2&0\\0&0&1\end{smallmatrix}\right)$ & $2$ & $\left(\begin{smallmatrix}0&0\\2&0\end{smallmatrix}\right)$ & yes\\
$1432$ & $\left(\begin{smallmatrix}1&0&0\\0&1&1\\0&1&0\end{smallmatrix}\right)$ & $4$ & $\left(\begin{smallmatrix}0&0\\2&1\end{smallmatrix}\right)$ & yes\\
$3214$ & $\left(\begin{smallmatrix}0&1&0\\1&1&0\\0&0&1\end{smallmatrix}\right)$ & $4$ & $\left(\begin{smallmatrix}1&0\\2&0\end{smallmatrix}\right)$ & yes\\
$3412$ & $\left(\begin{smallmatrix}0&1&0\\1&0&1\\0&1&0\end{smallmatrix}\right)$ & $4$ & $\left(\begin{smallmatrix}1&0\\2&1\end{smallmatrix}\right)$ & $=w$\\
$3421$ & $\left(\begin{smallmatrix}0&1&0\\0&1&1\\1&0&0\end{smallmatrix}\right)$ & $4$ & $\left(\begin{smallmatrix}1&0\\\mathbf{3}&1\end{smallmatrix}\right)$ & no\\
$4312$ & $\left(\begin{smallmatrix}0&0&1\\1&1&0\\0&1&0\end{smallmatrix}\right)$ & $4$ & $\left(\begin{smallmatrix}1&\mathbf{1}\\2&1\end{smallmatrix}\right)$ & no\\
$4321$ & $\left(\begin{smallmatrix}0&0&1\\0&2&0\\1&0&0\end{smallmatrix}\right)$ & $2$ & $\left(\begin{smallmatrix}1&\mathbf{1}\\\mathbf{3}&1\end{smallmatrix}\right)$ & no\\
\bottomrule
\end{tabular}
\end{center}
By Lemma~\ref{lem:corner}, a maximal element lies below $w$ iff its corner sums are at most those of $w$; the entries in bold are the ones that exceed them. So $w$ has four extremal pairs, with $x=1324,1432,3214,3412$, while its Bruhat interval has $14$ elements.
\end{exa}

\section{The recursion and the coefficients}\label{sec:recursion}

Pick $s\in D_L(w)$ and put $v=sw$. For extremal $x\le w$ we have $sx<x$, and the standard recursion \cite{KL} reads
\begin{equation}\label{eq:rec}
P_{x,w}=P_{sx,v}+q\,P_{x,v}-\sum_{z\in M}\mu(z,v)\,q^{(\ell(w)-\ell(z))/2}P_{x,z},\qquad M=\{z<v:\ sz<z,\ \mu(z,v)\ne0\}.
\end{equation}
Each $P_{y,u}$ on the right is looked up as $P_{\hat y,u}$ (with hats for the descents of $u$) in the sorted list of extremal pairs of $u$. If $\hat y$ is not in the list then $y\not\le u$ and $P_{y,u}=0$. By the descent theorem \cite[(2.3.e)]{KL}, if $\mu(z,v)\ne0$ and $\ell(v)-\ell(z)>1$ then $D_L(v)\subseteq D_L(z)$ and $D_R(v)\subseteq D_R(z)$. So every $z\in M$ is either a Bruhat cover of $v$, with $\mu=1$, or extremal for $v$, with $\mu(z,v)$ read off the stored polynomial. Elements are processed by increasing length, and equal polynomials share an identifier in a hash table.

\begin{exa}\label{exa:recursion}
Continue Example~\ref{exa:tables} with $w=3412$ and $s=s_2$, so $v=s_2w=2413$. Here $D_L(v)=\{s_1,s_3\}$ and $D_R(v)=\{s_2\}$, and the only extremal pair of $v$ is $(v,v)$; there is no $z<v$ with $s_2z<z$ and $\mu(z,v)\ne0$, so $M=\emptyset$ and \eqref{eq:rec} reads $P_{x,w}=P_{sx,v}+q\,P_{x,v}$. For the four extremal $x$:
\begin{itemize}
\item $x=1324$: $sx=1234$, and both $\widehat{1234}$ and $\widehat{1324}$ (for the descents of $v$) equal $2413=v$. So $P_{x,w}=1+q$.
\item $x=1432$: $sx=1423$ with $\widehat{1423}=v$, so $P_{sx,v}=1$. But $\widehat{1432}=2431$ is longer than $v$, so it is not in the list of $v$, and $P_{x,v}=0$. So $P_{x,w}=1$.
\item $x=3214$ and $x=3412$ behave in the same way, with $\hat x=4213$ and $4321$ not below $v$, and $P_{x,w}=1$.
\end{itemize}
\end{exa}

\subsection{Coefficients}\label{sec:coeff}
By \eqref{eq:duality} and Lemma~\ref{lem:descent},
\begin{equation}\label{eq:coeff}
c_{\lambda,w}=\sum_{v\le w}P_{v,w}(1)\,\phi^\lambda(v)=\sum_{x\text{ extremal for }w}P_{x,w}(1)\,G_\lambda(I,J,x),\qquad G_\lambda(I,J,x)=\sum_{v\in W_IxW_J}\phi^\lambda(v).
\end{equation}
$G_\lambda$ depends on $w$ only through $(I,J)$. For each of the $169{,}438$ classes we run once through $S_{10}$ by depth-first search. The double coset of each $v$ is identified by an additive hash of its table, updated one position at a time. The number of $v$ assigned to each double coset is then checked against the size formula of Section~\ref{sec:extremal}, which rules out hash collisions. For $n\le10$ all values $\phi^\lambda(v)$ are integers, and $|c_{\lambda,w}|\le n!\,\max|\phi^\lambda|\,\max_{x,w}P_{x,w}(1)<2^{62}$, so all arithmetic is exact in $64$-bit integers.

\begin{exa}\label{exa:coeff}
For $w=3412$, the four double cosets of Example~\ref{exa:tables} that lie below $w$ give the following values of $G_\lambda(I,J,x)$, for $\lambda=(4),(3,1),(2,2),(2,1,1),(1^4)$:
\[
G(1324)=(0,0,0,2,0),\quad G(1432)=G(3214)=(0,6,0,-2,0),\quad G(3412)=(8,-8,4,0,0).
\]
With the weights $P_{x,w}(1)=2,1,1,1$ from Example~\ref{exa:recursion}, \eqref{eq:coeff} gives
\[
c_{\cdot,3412}=2\,(0,0,0,2,0)+2\,(0,6,0,-2,0)+(8,-8,4,0,0)=(8,4,4,0,0).
\]
The individual $G_\lambda$ can be negative, and for $\lambda=(2,1,1)$ the coefficient vanishes only because of the weight $P_{1324,3412}(1)=2$.
\end{exa}

\subsection{Graded coefficients}\label{sec:graded}
At generic $q$ the same double-coset reduction gives
\begin{equation}\label{eq:qcoeff}
c_{\lambda,w}(q)=\sum_{x\text{ extremal for }w}P_{x,w}(q)\,G_\lambda(I,J,x)(q),\qquad G_\lambda(I,J,x)(q)=\sum_{v\in W_IxW_J}\phi^\lambda_q(T_v),
\end{equation}
with the same depth-first runs and table hashes as at $q=1$; only the accumulated quantity is now a vector of polynomials. The traces $\phi^\lambda_q(T_v)$ come from the class polynomials of Geck and Pfeiffer \cite[Ch.~8]{GP}. For every trace $\tau$ of $H_n$: if $\ell(sus)=\ell(u)$ then $\tau(T_u)=\tau(T_{sus})$; if $\ell(sus)=\ell(u)-2$ then $\tau(T_u)=(q-1)\tau(T_{su})+q\,\tau(T_{sus})$; and every $u$ reaches a minimal-length element of its class by cyclic shifts that do not increase length, all minimal-length elements of a class having the same trace \cite[Thms.~3.2.9, 8.2.3]{GP}. So $T_v\equiv\sum_\rho f_{v,\rho}(q)\,T_{c_\rho}$ modulo $[H_n,H_n]$ with integer polynomials $f_{v,\rho}$, computed by increasing length. For $h\in H_n$ write $\ch(h)=\sum_a\chi^a_q(h)\,s_a$, where $\chi^a_q$ is the irreducible character of $H_n$ that specialises to $\chi^a$ at $q=1$; the index and sign representations are $T_s\mapsto q$ and $T_s\mapsto-1$, with characters $\chi^{(n)}_q$ and $\chi^{(1^n)}_q$. Since $\chi^a_q=\sum_\lambda K_{a,\lambda}\phi^\lambda_q$ and $h_\lambda=\sum_aK_{a,\lambda}s_a$, we get $\ch(h)=\sum_\lambda\phi^\lambda_q(h)\,h_\lambda$: the coefficient of $h_\lambda$ in $\ch(h)$ is $\phi^\lambda_q(h)$. At the minimal-length element $c_\rho$, a product of Coxeter elements of consecutive blocks of sizes $\rho_1,\rho_2,\dots$, one has $\ch(T_{c_\rho})=\prod_k\sum_j(-1)^jq^{\rho_k-1-j}s_{(\rho_k-j,1^j)}$ \cite{Ram}, which at $q=1$ is the power sum $p_\rho$. Expanding in the $h$ basis gives $\phi^\lambda_q(T_{c_\rho})$, exactly, in rational arithmetic. Since $\deg\phi^\lambda_q(T_v)\le\ell(v)$ and $\deg P_{v,w}\le(\ell(w)-\ell(v)-1)/2$ for $v<w$, while $P_{w,w}=1$, every $c_{\lambda,w}(q)$ has degree at most $\ell(w)$. For $n=10$ the traces and the output each take about $56$\,GB, and the shapes are processed in batches. On $192$ cores the Kazhdan--Lusztig table takes about $100$\,s in both the $q=1$ and the graded run, while the graded coefficient step takes about $38$ times as long as at $q=1$ ($84$ minutes against $131$\,s), since each accumulated $G_\lambda$ is a polynomial of degree at most $45$ instead of an integer.

\begin{exa}\label{exa:traces}
In $S_3$ the classes are represented by $c_{(1^3)}=e$, $c_{(2,1)}=s_1$ and $c_{(3)}=s_1s_2$, and
\[
\ch(T_e)=h_1^3,\qquad \ch(T_{s_1})=(qs_2-s_{11})s_1=(1+q)h_{21}-h_{111},
\]
\[
\ch(T_{s_1s_2})=q^2s_3-qs_{21}+s_{111}=(1+q+q^2)h_3-(2+q)h_{21}+h_{111}.
\]
The other elements are reduced to these. The elements $213=s_1$ and $132=s_2$ both have minimal length in their class, and $312$ is a cyclic shift of $231=s_1s_2$ of the same length, so they share the traces above. For $w_0=321$, conjugating by $s=s_1$ gives $sw_0s=s_2$, of length $\ell(w_0)-2$, so $\tau(T_{321})=(q-1)\tau(T_{312})+q\,\tau(T_{132})$ for every trace $\tau$. The values $\phi^\lambda_q(T_v)$ are:
\begin{center}
\small
\begin{tabular}{lcccc}
\toprule
$v$ & $123$ & $213,\ 132$ & $231,\ 312$ & $321$\\
\midrule
$\phi^{(3)}_q(T_v)$ & $0$ & $0$ & $1+q+q^2$ & $q^3-1$\\
$\phi^{(2,1)}_q(T_v)$ & $0$ & $1+q$ & $-2-q$ & $2$\\
$\phi^{(1^3)}_q(T_v)$ & $1$ & $-1$ & $1$ & $-1$\\
\bottomrule
\end{tabular}
\end{center}
At $q=1$ the rows are the class functions $\phi^\lambda$. The lower interval of $w=231$ is $\{123,213,132,231\}$, with all $P_{v,w}=1$, so
\[
c_{\cdot,231}(q)=\bigl(1+q+q^2,\ q,\ 0\bigr)\qquad(\lambda=(3),(2,1),(1^3)),
\]
which specialises to the values $3,1,0$ of Example~\ref{exa:stability}. Summing over all of $S_3$ gives $c_{\cdot,321}(q)=\bigl((1+q)(1+q+q^2),0,0\bigr)=\bigl([3]_q!,0,0\bigr)$, the graded form of Lemma~\ref{lem:J}(3). The same table comes out of the induced modules of check~12 below, which use neither class polynomials nor \cite{Ram}.
\end{exa}

\section{Verification and results}\label{sec:checks}

\begin{theorem}\label{thm:n10}
For every $\lambda\vdash 10$ and every $w\in S_{10}$, $c_{\lambda,w}(q)\in\Z_{\ge0}[q]$, and it is symmetric and unimodal about $q^{\ell(w)/2}$.
\end{theorem}

The computation of Sections~\ref{sec:extremal} and~\ref{sec:recursion} gives every $c_{\lambda,w}(q)$. None has a negative coefficient, none is asymmetric and none fails to be unimodal. $S_{10}$ has $981{,}174$ distinct Kazhdan--Lusztig polynomials, of degree at most $12$; the largest coefficient of any $c_{\lambda,w}(q)$ is $35{,}419{,}702$. Table~\ref{tab:n10} gives the support and range of the non-zero values at $q=1$. The ungraded computation took $279$ seconds on a $192$-core machine, with a peak memory of $15.5$\,GB; the graded one took $1$\,h\,$31$\,min, with a peak of $125$\,GB, almost all of it in accumulating \eqref{eq:qcoeff}. All arithmetic is exact in $64$-bit integers, by the a priori bounds at the end of this section. We ran the following checks.
\begin{figure}[t]
\centering
\begin{tikzpicture}[stage/.style={draw,thick,rounded corners,align=center,minimum height=1.3cm,text width=2.9cm,font=\small},
  chk/.style={draw,align=center,font=\footnotesize,fill=gray!8,text width=2.3cm,minimum height=0.9cm},
  >=stealth]
  \node[stage] (A) at (0,0) {extremal pairs\\(tables and\\corner sums,\\Lemma~\ref{lem:corner})};
  \node[stage] (B) at (3.7,0) {KL recursion\\$P_{x,w}$\\(Section~\ref{sec:recursion})};
  \node[stage] (C) at (7.4,0) {double-coset sums\\$G_\lambda(I,J,x)$\\(eq.~\eqref{eq:coeff})};
  \node[stage] (D) at (11.1,0) {coefficients\\$c_{\lambda,w}$};
  \draw[->,thick] (A) -- (B); \draw[->,thick] (B) -- (C); \draw[->,thick] (C) -- (D);
  \node[chk] (c1) at (1.85,2.1) {1.\ sanity of $P_{x,w}$};
  \node[chk] (c2) at (1.85,-2.1) {2.\ Coxeter3,\\all pairs};
  \node[chk] (c9) at (5.55,-2.1) {9.\ smaller $n$,\\full intervals};
  \node[chk] (c34) at (9.25,2.1) {3.\ $J$-values\\4.\ symmetries};
  \node[chk] (c5) at (9.25,-2.1) {5.\ KL inversion};
  \node[chk] (c678) at (12.3,2.1) {6.\ two columns\\7.\ codominant\\8.\ coproduct};
  \draw[->,gray] (c1) -- (B);
  \draw[->,gray] (c2) -- (A); \draw[->,gray] (c2) -- (B);
  \draw[->,gray] (c9) -- (B); \draw[->,gray] (c9) -- (D);
  \draw[->,gray] (c34) -- (D); \draw[->,gray] (c5) -- (D); \draw[->,gray] (c678) -- (D);
\end{tikzpicture}
\caption{The computation of Theorem~\ref{thm:n10} (top row) and the stages tested by the checks of Section~\ref{sec:checks}. Check~2 compares every extremal pair, so it also tests that the enumerated pairs are the right ones. Checks~3, 4 and~6--8 test the output against facts proved independently of the computation.}\label{fig:pipeline}
\end{figure}
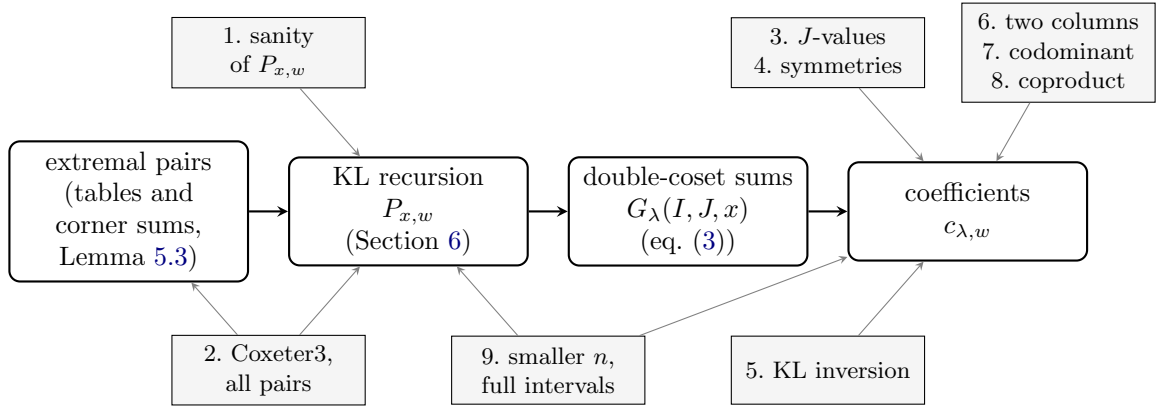

\begin{enumerate}
\item \emph{Sanity of the polynomials.} Every stored $P_{x,w}$ has constant term $1$ and non-negative coefficients, $P_{w,w}=1$, and $\deg P_{x,w}\le(\ell(w)-\ell(x)-1)/2$ for $x<w$.
\item \emph{Independent software.} For every $n\le 10$, every extremal pair, including all $1{,}252{,}173{,}030$ in $S_{10}$, was compared with du~Cloux's \textsf{Coxeter3} \cite{Cox3}, with no discrepancies. This took $2.4$ hours on $190$ cores.
\item \emph{$J$-values.} For all $2^9=512$ compositions $\alpha$ of $10$ and all $\lambda$, $c_{\lambda,w_0(\alpha)}=\alpha!\,\delta_{\lambda,\mathrm{sort}(\alpha)}$ (Lemma~\ref{lem:J}(3)).
\item \emph{Symmetries.} $c_{\lambda,w}=c_{\lambda,w^{-1}}=c_{\lambda,w_0ww_0}$ for all $\lambda,w$. The left-descent recursion does not build these in.
\item \emph{Kazhdan--Lusztig inversion.} For $2000$ random $u\in S_{10}$ and every $\lambda$, the coefficient of $x_{1,u(1)}\cdots x_{n,u(n)}$ in $\sum_wc_{\lambda,w}\Imm_w$ equals $\phi^\lambda(u)$. This tests the coefficients against the inversion formula, using the stored polynomials at all pairs $x\le w$, not only extremal ones. The check is exact over $\Z$: the sum $\sum|P(1)|\,|c_{\lambda,w}|$ bounds every partial sum and is below $2^{41}$, so no $64$-bit partial sum wraps around, and the sums were also accumulated modulo a prime $p<2^{31}$.
\item \emph{Two columns.} For $\lambda_1\le2$, $c_{\lambda,u}=0$ whenever $u$ contains $321$, a theorem that follows from Fan and Green's description of the Temperley--Lieb quotient \cite{FG} (see \cite[Lemma~8.1]{McN}).
\item \emph{Codominant elements.} For each of the $16{,}796$ codominant ($312$-avoiding) $w$, $c_{\lambda,w}=\Imm_\lambda(A_w)$, where $A_w$ is the $0/1$ matrix with $(A_w)_{ij}=[j\le\max(w(1),\dots,w(i))]$ \cite[Prop.~7.1]{McN}. The right-hand side was evaluated without any Kazhdan--Lusztig computation. Check~13 covers the same $w$ for every $q$, against reference values that do not depend on \cite{McN}.
\item \emph{Coproduct rule.} For all $799{,}475$ decomposable $w=w_1\times w_2$ and all $\lambda$ ($33{,}577{,}950$ coefficients), $c_{\lambda,w}$ agrees with Lemma~\ref{lem:coproduct} applied to independent runs for $S_m$, $m<10$.
\item \emph{Smaller $n$.} For $n\le9$ the engine reproduces every $c_{\lambda,w}$ of an independent full-interval engine, which stores all $10{,}501{,}351{,}657$ Bruhat pairs of $S_9$. For small $n$ every stored $P_{x,w}$ agrees with a reference implementation and with SageMath.
\item \emph{Graded output at $q=1$.} $c_{\lambda,w}(1)$ equals the value of the ungraded run for every $\lambda,w$.
\item \emph{Graded symmetries and stability.} $c_{\lambda,w}(q)=c_{\lambda,w^{-1}}(q)=c_{\lambda,w_0ww_0}(q)$ for all $\lambda,w$, and Lemma~\ref{lem:qstability} holds exactly against a separate run for $S_9$. The method builds in neither.
\item \emph{Traces from the Hecke relations.}
  \begin{enumerate}
  \item \emph{Induced modules, $n\le6$.} For $n\le6$ and $q=2,3$, every $c_{\lambda,w}(q)$ equals the trace of $\sum_vP_{v,w}(q)T_v$ on the induced modules $M^\mu=H_n\otimes_{H_\mu}(\text{index})$, built from the defining relations alone, without class polynomials or character formulas, and with Kazhdan--Lusztig polynomials from a separate implementation of the standard recursion. For $n\le5$ the comparison is exact over $\Z$; for $n=6$ (modules of dimension up to $720$) it is made modulo the primes $2^{20}-3$ and $2^{20}-5$, with no mismatch.
  \item \emph{Induced modules, $n=10$.} The six modules $M^\mu$ with $\mu=(9,1)$, $(8,2)$, $(8,1,1)$, $(7,3)$, $(6,4)$, $(5,5)$, of dimension at most $252$, are traversed depth first along a spanning tree of the weak order. For every $v\in S_{10}$ the trace $\eta^\mu_q(T_v)$ agrees with $\sum_\lambda N_{\mu\lambda}\phi^\lambda_q(T_v)$ computed from the class polynomials, where $h_\mu=\sum_\lambda N_{\mu\lambda}m_\lambda$.
  \item \emph{Seminormal modules, $n=10$.} Every trace column is tested in the same way through the $42$ irreducible characters. Hoefsmit's seminormal form \cite{Hoe74,Ram97} gives explicit matrices for the $T_s$ on the irreducible modules $S^\mu$, of dimension at most $768$, from the combinatorics of standard tableaux; they satisfy the quadratic and braid relations (checked for every shape of size at most $8$). Traversed in the same way, $\chi^\mu_q(T_v)$ agrees with $\sum_\lambda K_{\mu\lambda}\phi^\lambda_q(T_v)$ for every $\mu\vdash10$ and every $v\in S_{10}$, where $s_\mu=\sum_\lambda K_{\mu\lambda}m_\lambda$. Since the Kostka matrix is unitriangular, this tests each of the $42$ traces $\phi^\lambda_q(T_v)$ at every $v$.
  \item \emph{Sampled $w$, $n=10$.} For $300$ random $w\in S_{10}$ containing $312$, with $P_{v,w}=P_{v^*,w}$ from the extremal pairs listed by \textsf{Coxeter3} ($v^*$ the maximum of $W_{D_L(w)}vW_{D_R(w)}$, Lemma~\ref{lem:descent}), $\sum_\lambda N_{\mu\lambda}c_{\lambda,w}(q)$ agrees with $\sum_vP_{v,w}(q)\eta^\mu_q(T_v)$ for the six $\mu$, and the stored $c_{\lambda,w}(q)$ agree with $\sum_vP_{v,w}(q)\phi^\lambda_q(T_v)$ for all $42$ shapes. The last comparison takes the traces $\phi^\lambda_q(T_v)$ as given and tests only the step \eqref{eq:qcoeff}; the others use no class polynomial. The comparisons at $n=10$ are made at $q=2,3$ modulo the prime $2^{31}-19$, with no mismatch.
  \end{enumerate}
\item \emph{Chromatic quasisymmetric functions.} For each of the $16{,}796$ codominant $w$ and all $42$ shapes, $c_{\lambda,w}(q)$ equals the coefficient of $e_\lambda$ in the chromatic quasisymmetric function of the indifference graph of $w$, computed from proper colourings without any Kazhdan--Lusztig or Hecke algebra input.
\end{enumerate}
Check~2 covers the polynomials, and checks~3--8 the coefficient step \eqref{eq:coeff}, partly with no Kazhdan--Lusztig input at all (checks~3, 6 and~7). Checks~10--13 cover the graded step \eqref{eq:qcoeff}; check~13 is independent of the whole pipeline. Figure~\ref{fig:pipeline} shows which stage each check tests, and Table~\ref{tab:depend} which code each check shares with the computation.

\begin{table}[ht]
\small
\begin{tabular}{rp{2.3cm}p{6.4cm}p{2.6cm}}
\toprule
check & code shared & compared with & stages tested\\
\midrule
1 & E, K & constant term and degree bound of KL polynomials & K\\
2 & none & \textsf{Coxeter3} (independent C\texttt{++}), every extremal pair and its $x$-list & E, K\\
3 & E, K, $\Phi$, S & Lemma~\ref{lem:J}(3), a theorem & $\Phi$, S\\
4 & E, K, $\Phi$, S & symmetries of $P_{v,w}$, not built into the recursion & K, S\\
5 & K, $\Phi$ & identity \eqref{eq:duality} at all pairs $x\le w$, without the double-coset sums & S\\
6 & E, K, $\Phi$, S & Fan--Green \cite{FG}, a theorem & K, S\\
7 & $\Phi$ & $\Imm_\lambda(A_w)$, no Kazhdan--Lusztig input & E, K, S (codominant $w$)\\
8 & E, K, $\Phi$, S & Lemma~\ref{lem:coproduct}, separate runs for $S_m$, $m<10$ & consistency in $n$\\
9 & $\Phi$ & full-interval engine (separate code); for small $n$ a reference implementation and SageMath & E, K, S ($n\le9$)\\
10 & E, K, S & the ungraded run, which uses $\Phi$ instead of T & T at $q=1$\\
11 & E, K, S, T & symmetries and Lemma~\ref{lem:qstability}, not built in & T\\
12 & none; at $n=10$ Lemma~\ref{lem:descent}, which S also uses, and T as input to one comparison & traces on induced modules built from the Hecke relations, and at $n=10$ on Hoefsmit's seminormal modules; KL polynomials by the standard recursion ($n\le6$) or \textsf{Coxeter3} ($n=10$) & all ($n\le6$); at $n=10$, T at every $v$ (all $42$ characters) and all stages at $300$ sampled $w$\\
13 & none & chromatic quasisymmetric functions from colourings, via the dictionary of \cite{CHSS} & all (codominant $w$)\\
\bottomrule
\end{tabular}
\caption{Code shared by each check of Section~\ref{sec:checks} with the computation. The stages are: E, the enumeration of extremal pairs (Lemma~\ref{lem:corner}); K, the Kazhdan--Lusztig recursion (Section~\ref{sec:recursion}); $\Phi$, the table of $\phi^\lambda$ on conjugacy classes; S, the double-coset sums \eqref{eq:coeff} and \eqref{eq:qcoeff}; T, the traces $\phi^\lambda_q(T_v)$ from class polynomials (Section~\ref{sec:graded}).}\label{tab:depend}
\end{table}

\emph{Independence.} At $n=10$ the stages E, K, S and T (at $q=1$) are each tested by a check that shares none of their code: E and K by check~2, S by checks~5 and~7, and T at $q=1$ by check~10. The table $\Phi$ is tested against a theorem by check~3, and recomputed from scratch in check~12 for $n\le6$. Agreement between programs that share code could hide a common error, and the table shows where this matters. At $q=2$ and $q=3$, modulo $2^{31}-19$, T is tested independently at every $v\in S_{10}$ by check~12, through all $42$ irreducible characters; it is also tested by check~11. These are evaluations at two values of $q$, not identities of polynomials in $q$. The weakest point is S at generic $q$ for non-codominant $w$, the assembly \eqref{eq:qcoeff} of the polynomials $c_{\lambda,w}(q)$ from the traces. At $q=1$ the same step is tested independently by check~5, at $2000$ random $u$, and by check~7, at all $16{,}796$ codominant $w$. At generic $q$ it is tested by check~11 and by the symmetry and unimodality counts, which the method does not build in, and independently only end to end, at $300$ sampled $w$ and at $q=2,3$ (check~12). An exhaustive independent test would need the sum over all $\sim10^{12}$ Bruhat pairs of $S_{10}$. The checks that fix the conventions for $T_s$, $C'_w$ and the Frobenius map independently are check~12, from the defining relations ($n\le6$ for every $w$, $n=10$ as above), and check~13, from colourings (codominant $w$, $n\le10$); a convention error would have to pass both.

\emph{Integer safety.} No intermediate value can leave the range of signed $64$-bit integers, by three bounds that are a priori except for two numbers read off the finished tables, as explained below. (i) Let $A(u)$ be the sum of the absolute values of the coefficients of all the class polynomials $f_{u,\rho}$. No element of length at most $2$ admits a conjugation $u\mapsto sus$ with $\ell(sus)=\ell(u)-2$, so these elements have $A(u)=1$; cyclic shifts do not change $A$, and the second rule of Section~\ref{sec:graded} gives $A(u)\le2A(su)+A(sus)$. Hence $A(u)\le M_{\ell(u)}$ with $M_0=M_1=M_2=1$ and $M_\ell=2M_{\ell-1}+M_{\ell-2}$, and the same bounds every partial sum of the recursion. For $n=10$, $M_{45}\approx3.48\times10^{16}<2^{61}$. (ii) Every partial sum of $\phi^\lambda_q(T_v)=\sum_\rho f_{v,\rho}\,\phi^\lambda_q(T_{c_\rho})$ is at most $M_{45}\beta$, where $\beta=120$ is the largest sum of the absolute values of the coefficients of a $\phi^\lambda_q(T_{c_\rho})$, computed in rational arithmetic. So it is at most $4.17\times10^{18}<2^{63}$. (iii) In \eqref{eq:qcoeff} the double-coset sums are at most $n!\,m$, and every partial sum is at most $m\cdot n!\cdot\max_{x,w}P_{x,w}(1)=6.68\times10^{13}$, where $m$ is the largest sum of the absolute values of the coefficients of a trace $\phi^\lambda_q(T_v)$. The numbers $m$ and $\max P_{x,w}(1)$ are read off the finished tables, which are exact by (i), (ii) and check~2. Degrees are guarded separately: a coefficient beyond the stored degree stops the run. The bound at $q=1$ in Section~\ref{sec:coeff} is the analogue of (iii). For comparison, the largest coefficient that occurs is $35{,}419{,}702$.

\begin{table}[ht]
\small
\begin{tabular}{lrr@{\qquad}lrr}
\toprule
$\lambda$ & $\#\{w:c_{\lambda,w}>0\}$ & range & $\lambda$ & $\#\{w:c_{\lambda,w}>0\}$ & range\\
\midrule
$10$ & 2{,}829{,}325 & $10$--$384{,}025{,}920$ & $4\,3^{2}$ & 1{,}446{,}300 & $3$--$192{,}364$\\
$9\,1$ & 3{,}376{,}010 & $8$--$253{,}386{,}624$ & $4\,3\,2\,1$ & 1{,}844{,}393 & $1$--$209{,}952$\\
$8\,2$ & 2{,}887{,}431 & $16$--$109{,}171{,}840$ & $4\,3\,1^{3}$ & 659{,}730 & $2$--$17{,}984$\\
$8\,1^{2}$ & 3{,}431{,}287 & $7$--$37{,}088{,}768$ & $4\,2^{3}$ & 397{,}450 & $4$--$17{,}248$\\
$7\,3$ & 2{,}847{,}535 & $21$--$51{,}336{,}714$ & $4\,2^{2}\,1^{2}$ & 527{,}279 & $2$--$11{,}572$\\
$7\,2\,1$ & 3{,}441{,}650 & $7$--$25{,}228{,}512$ & $4\,2\,1^{4}$ & 103{,}683 & $2$--$736$\\
$7\,1^{3}$ & 2{,}511{,}423 & $6$--$1{,}264{,}516$ & $4\,1^{6}$ & 5{,}358 & $3$--$160$\\
$6\,4$ & 2{,}803{,}637 & $24$--$24{,}591{,}232$ & $3^{3}\,1$ & 500{,}625 & $2$--$17{,}020$\\
$6\,3\,1$ & 3{,}362{,}247 & $12$--$11{,}369{,}056$ & $3^{2}\,2^{2}$ & 308{,}894 & $2$--$11{,}588$\\
$6\,2^{2}$ & 2{,}458{,}248 & $8$--$2{,}578{,}752$ & $3^{2}\,2\,1^{2}$ & 367{,}081 & $1$--$4{,}836$\\
$6\,2\,1^{2}$ & 2{,}709{,}706 & $4$--$867{,}808$ & $3^{2}\,1^{4}$ & 49{,}280 & $2$--$440$\\
$6\,1^{4}$ & 680{,}307 & $5$--$48{,}560$ & $3\,2^{3}\,1$ & 133{,}373 & $1$--$1{,}340$\\
$5^{2}$ & 2{,}534{,}300 & $15$--$9{,}236{,}960$ & $3\,2^{2}\,1^{3}$ & 55{,}642 & $1$--$232$\\
$5\,4\,1$ & 3{,}064{,}042 & $6$--$6{,}169{,}664$ & $3\,2\,1^{5}$ & 7{,}421 & $1$--$48$\\
$5\,3\,2$ & 2{,}486{,}584 & $4$--$1{,}581{,}664$ & $3\,1^{7}$ & 309 & $2$--$8$\\
$5\,3\,1^{2}$ & 2{,}542{,}958 & $4$--$464{,}880$ & $2^{5}$ & 5{,}662 & $2$--$32$\\
$5\,2^{2}\,1$ & 1{,}954{,}375 & $3$--$240{,}900$ & $2^{4}\,1^{2}$ & 8{,}640 & $1$--$16$\\
$5\,2\,1^{3}$ & 839{,}778 & $3$--$34{,}700$ & $2^{3}\,1^{4}$ & 2{,}167 & $1$--$8$\\
$5\,1^{5}$ & 76{,}455 & $4$--$1{,}200$ & $2^{2}\,1^{6}$ & 301 & $1$--$4$\\
$4^{2}\,2$ & 1{,}660{,}725 & $4$--$482{,}904$ & $2\,1^{8}$ & 25 & $1$--$2$\\
$4^{2}\,1^{2}$ & 1{,}807{,}360 & $3$--$195{,}552$ & $1^{10}$ & 1 & $1$--$1$\\
\bottomrule
\end{tabular}
\caption{Support size and range of the non-zero coefficients $c_{\lambda,w}$, $n=10$.}\label{tab:n10}
\end{table}

\section{Consequences for \texorpdfstring{$n\le10$}{n<=10}}\label{sec:consequences}

\begin{proof}[Proof of Theorem~\ref{thm:main} and Corollary~\ref{cor:cones}]
Theorem~\ref{thm:n10} and Corollary~\ref{cor:down} give Theorem~\ref{thm:main}. Then $\CKL=\CJ$ by Theorem~\ref{thm:cones}, and the chain $\CKL\subseteq\CTNN\subseteq\CJ$ gives $\CKL=\CTNN=\CJ$.
\end{proof}

\subsection{At generic $q$}
Positivity at $q$ is coefficientwise, so the analogues of the cones are $\N[q]$-semimodules of traces of $H_n$ with values in $\Z[q]$. Let $\CKL(q)$ be the set of traces $\tau$ with $\tau(q^{\ell(w)/2}C'_w)\in\N[q]$ for all $w\in S_n$, and let $\CJ(q)=\sum_\lambda\N[q]\,\phi^\lambda_q$. Since $q^{\ell(w_0(\mu))/2}C'_{w_0(\mu)}=\sum_{v\in S_\mu}T_v$ and $\phi^\lambda_q$ takes the value $\delta_{\lambda\mu}[\mu]_q!$ there ($[\mu]_q!=\prod_i[\mu_i]_q!$; Haiman \cite[Prop.~4.1(1)]{Hai}), the coordinates of $\tau=\sum_\lambda a_\lambda\phi^\lambda_q$ are $a_\mu=\tau\bigl(\sum_{v\in S_\mu}T_v\bigr)/[\mu]_q!$, and $\tau\in\CJ(q)$ iff all of them lie in $\N[q]$. The positivity part of Conjecture~\ref{conj:Hq} is the inclusion $\CJ(q)\subseteq\CKL(q)$, and Theorem~\ref{thm:main} gives it for $n\le10$.

Unlike at $q=1$, this inclusion need not be an equality, and the evaluations at the $w_0(\mu)$ do not detect membership in $\CKL(q)$. Let $\CJ^{\mathrm{ev}}(q)$ be the set of traces with $\tau(\sum_{v\in S_\mu}T_v)\in\N[q]$ for all $\mu$. Then $\CKL(q)\subseteq\CJ^{\mathrm{ev}}(q)$ for every $n$, and for $n\le10$
\[
\CJ(q)\subseteq\CKL(q)\subseteq\CJ^{\mathrm{ev}}(q).
\]
The first inclusion is strict for $n=2$: the trace $(1-q+q^2)\phi^{(2)}_q$ takes the values $0$ at $T_e$ and $(1-q+q^2)(1+q)=1+q^3$ at $q^{1/2}C'_s=T_e+T_s$, so it lies in $\CKL(q)$, but its coordinate $1-q+q^2$ is not in $\N[q]$. The second is strict for every $n\ge3$: the trace $\tau=(1-q+q^2)\phi^{(2,1^{n-2})}_q$ has a single non-zero evaluation at the $w_0(\mu)$, namely $(1-q+q^2)[2]_q=1+q^3$, so $\tau\in\CJ^{\mathrm{ev}}(q)$; but by Lemma~\ref{lem:qstability} and Example~\ref{exa:traces}, $c_{(2,1^{n-2}),\,231\times e_{n-3}}(q)=c_{(2,1),231}(q)=q$, so $\tau$ takes the value $q-q^2+q^3$ there and $\tau\notin\CKL(q)$. At $q=1$ the difference disappears, because a real number $a_\mu$ with $\mu!\,a_\mu\ge0$ is itself non-negative. So the graded form of Corollary~\ref{cor:cones} is the inclusion $\CJ(q)\subseteq\CKL(q)$ for $n\le10$, and nothing more. This inclusion still lifts every relation of the normalised order to $q$:

\begin{corollary}\label{cor:qorder}
Let $n\le10$ and let $a,b\vdash n$ satisfy the criterion of Theorem~\ref{thm:order}. Then $\chi^a_q(q^{\ell(w)/2}C'_w)/f^a-\chi^b_q(q^{\ell(w)/2}C'_w)/f^b$ has non-negative rational coefficients for every $w\in S_n$.
\end{corollary}

\begin{proof}
$\chi^a_q/f^a-\chi^b_q/f^b=\sum_\mu(K_{a,\mu}/f^a-K_{b,\mu}/f^b)\phi^\mu_q$, with non-negative coefficients by the criterion; evaluate at $q^{\ell(w)/2}C'_w$ and apply Theorem~\ref{thm:main}.
\end{proof}

\subsection{At $q=1$}
In terms of irreducible characters, for $\phi=\sum_a\gamma_a\chi^a$ we have $d_\phi(J_\mu)=\mu!\sum_a\gamma_aK_{a,\mu}$. So for $n\le10$, every linear inequality between irreducible-character immanants on $\TNN_n$ is decided by a finite check on Kostka numbers. The same holds for inequalities known on positive semidefinite matrices, since each $J_\mu$ is positive semidefinite; for normalised characters this observation is Stembridge's \cite[\S3]{Ste92} (see \cite[Prop.~6.5]{McN}).

\subsection{The order of normalised immanants}
For $a,b\vdash n$ write $a\succeq b$ if $N_a\ge N_b$ on $\TNN_n$.

\begin{proof}[Proof of Theorem~\ref{thm:order}]
$f^af^b(N_a-N_b)=d_\phi$ with $\phi=f^b\chi^a-f^a\chi^b$, and $d_\phi(J_\mu)=\mu!\,(f^bK_{a,\mu}-f^aK_{b,\mu})$ by Lemma~\ref{lem:J}. So the criterion is $\phi\in\CJ$, and $a\succeq b$ is $\phi\in\CTNN$. They agree by Corollary~\ref{cor:cones}.
\end{proof}

The necessity of the criterion holds for every $n$, since $J_\mu\in\TNN_n$; it is Stembridge's \cite[Prop.~3.1]{Ste92}, who notes that his conjecture gives the converse (see also \cite[Prop.~6.3]{McN}). The content of Theorem~\ref{thm:order} is the converse for $n\le10$.

Taking $\mu=b$ shows that $a\succeq b$ implies $K_{a,b}>0$, so $\succeq$ is contained in dominance order for every $n$. Figure~\ref{fig:hasse} shows the Hasse diagrams for $n\le7$, and Appendix~\ref{app:hasse} those for $n=8,9,10$.

\begin{exa}\label{exa:kostka}
Let $n=4$, $a=(3,1)$ and $b=(2,2)$, so $a$ dominates $b$, and $f^a=3$, $f^b=2$. For $\mu=(4),(3,1),(2,2),(2,1,1),(1^4)$,
\[
\frac{K_{a,\mu}}{f^a}=\Bigl(0,\tfrac13,\tfrac13,\tfrac23,1\Bigr),\qquad \frac{K_{b,\mu}}{f^b}=\Bigl(0,0,\tfrac12,\tfrac12,1\Bigr).
\]
The criterion for $a\succeq b$ fails at $\mu=(2,2)$, and indeed $N_a(J_{22})=4\cdot\frac13<N_b(J_{22})=4\cdot\frac12$. The criterion for $b\succeq a$ fails at $\mu=(3,1)$, where $N_b(J_{31})=0<N_a(J_{31})=2$. So $(3,1)$ and $(2,2)$ are incomparable in $\succeq$, which is the first failure of dominance in Theorem~\ref{thm:structure}(1).
\end{exa}

\begin{figure}[ht]
\centering
\resizebox{0.16\textwidth}{!}{\begin{tikzpicture}[every node/.style={font=\scriptsize,inner sep=1.5pt}]
  \node (p4) at (0.00,3.30) {$4$};
  \node (p3_1) at (-0.62,2.20) {$3\,1$};
  \node (p2_2) at (0.62,2.20) {$2^{2}$};
  \node (p2_1_1) at (0.00,1.10) {$2\,1^{2}$};
  \node (p1_1_1_1) at (0.00,0.00) {$1^{4}$};
  \draw[gray] (p2_1_1) -- (p1_1_1_1);
  \draw[gray] (p2_2) -- (p2_1_1);
  \draw[gray] (p3_1) -- (p2_1_1);
  \draw[gray] (p4) -- (p2_2);
  \draw[gray] (p4) -- (p3_1);
\end{tikzpicture}}\hfill
\resizebox{0.18\textwidth}{!}{\begin{tikzpicture}[every node/.style={font=\scriptsize,inner sep=1.5pt}]
  \node (p5) at (0.00,5.50) {$5$};
  \node (p4_1) at (0.00,4.40) {$4\,1$};
  \node (p3_2) at (0.00,3.30) {$3\,2$};
  \node (p3_1_1) at (-0.62,2.20) {$3\,1^{2}$};
  \node (p2_2_1) at (0.62,2.20) {$2^{2}\,1$};
  \node (p2_1_1_1) at (0.00,1.10) {$2\,1^{3}$};
  \node (p1_1_1_1_1) at (0.00,0.00) {$1^{5}$};
  \draw[gray] (p2_1_1_1) -- (p1_1_1_1_1);
  \draw[gray] (p2_2_1) -- (p2_1_1_1);
  \draw[gray] (p3_1_1) -- (p2_1_1_1);
  \draw[gray] (p3_2) -- (p2_2_1);
  \draw[gray] (p3_2) -- (p3_1_1);
  \draw[gray] (p4_1) -- (p3_2);
  \draw[gray] (p5) -- (p4_1);
\end{tikzpicture}}\hfill
\resizebox{0.24\textwidth}{!}{\begin{tikzpicture}[every node/.style={font=\scriptsize,inner sep=1.5pt}]
  \node (p6) at (0.00,6.60) {$6$};
  \node (p5_1) at (0.00,5.50) {$5\,1$};
  \node (p4_2) at (-0.62,4.40) {$4\,2$};
  \node (p4_1_1) at (-1.25,3.30) {$4\,1^{2}$};
  \node (p3_3) at (0.62,4.40) {$3^{2}$};
  \node (p3_2_1) at (0.00,3.30) {$3\,2\,1$};
  \node (p3_1_1_1) at (-0.62,2.20) {$3\,1^{3}$};
  \node (p2_2_2) at (1.25,3.30) {$2^{3}$};
  \node (p2_2_1_1) at (0.62,2.20) {$2^{2}\,1^{2}$};
  \node (p2_1_1_1_1) at (0.00,1.10) {$2\,1^{4}$};
  \node (p1_1_1_1_1_1) at (0.00,0.00) {$1^{6}$};
  \draw[gray] (p2_1_1_1_1) -- (p1_1_1_1_1_1);
  \draw[gray] (p2_2_1_1) -- (p2_1_1_1_1);
  \draw[gray] (p2_2_2) -- (p2_2_1_1);
  \draw[gray] (p3_1_1_1) -- (p2_1_1_1_1);
  \draw[gray] (p3_2_1) -- (p2_2_1_1);
  \draw[gray] (p3_2_1) -- (p3_1_1_1);
  \draw[gray] (p3_3) -- (p2_2_2);
  \draw[gray] (p3_3) -- (p3_2_1);
  \draw[gray] (p4_1_1) -- (p2_2_1_1);
  \draw[gray] (p4_1_1) -- (p3_1_1_1);
  \draw[gray] (p4_2) -- (p2_2_2);
  \draw[gray] (p4_2) -- (p3_2_1);
  \draw[gray] (p4_2) -- (p4_1_1);
  \draw[gray] (p5_1) -- (p3_3);
  \draw[gray] (p5_1) -- (p4_2);
  \draw[gray] (p6) -- (p5_1);
\end{tikzpicture}}\hfill
\resizebox{0.3\textwidth}{!}{\begin{tikzpicture}[every node/.style={font=\scriptsize,inner sep=1.5pt}]
  \node (p7) at (0.00,9.90) {$7$};
  \node (p6_1) at (0.00,8.80) {$6\,1$};
  \node (p5_2) at (0.00,7.70) {$5\,2$};
  \node (p5_1_1) at (-0.62,6.60) {$5\,1^{2}$};
  \node (p4_3) at (0.62,6.60) {$4\,3$};
  \node (p4_2_1) at (-0.62,5.50) {$4\,2\,1$};
  \node (p4_1_1_1) at (-1.25,3.30) {$4\,1^{3}$};
  \node (p3_3_1) at (0.62,5.50) {$3^{2}\,1$};
  \node (p3_2_2) at (0.00,4.40) {$3\,2^{2}$};
  \node (p3_2_1_1) at (0.00,3.30) {$3\,2\,1^{2}$};
  \node (p3_1_1_1_1) at (-0.62,2.20) {$3\,1^{4}$};
  \node (p2_2_2_1) at (1.25,3.30) {$2^{3}\,1$};
  \node (p2_2_1_1_1) at (0.62,2.20) {$2^{2}\,1^{3}$};
  \node (p2_1_1_1_1_1) at (0.00,1.10) {$2\,1^{5}$};
  \node (p1_1_1_1_1_1_1) at (0.00,0.00) {$1^{7}$};
  \draw[gray] (p2_1_1_1_1_1) -- (p1_1_1_1_1_1_1);
  \draw[gray] (p2_2_1_1_1) -- (p2_1_1_1_1_1);
  \draw[gray] (p2_2_2_1) -- (p2_2_1_1_1);
  \draw[gray] (p3_1_1_1_1) -- (p2_1_1_1_1_1);
  \draw[gray] (p3_2_1_1) -- (p2_2_1_1_1);
  \draw[gray] (p3_2_1_1) -- (p3_1_1_1_1);
  \draw[gray] (p3_2_2) -- (p2_2_2_1);
  \draw[gray] (p3_2_2) -- (p3_2_1_1);
  \draw[gray] (p3_3_1) -- (p3_2_2);
  \draw[gray] (p4_1_1_1) -- (p2_2_1_1_1);
  \draw[gray] (p4_1_1_1) -- (p3_1_1_1_1);
  \draw[gray] (p4_2_1) -- (p3_2_2);
  \draw[gray] (p4_2_1) -- (p4_1_1_1);
  \draw[gray] (p4_3) -- (p3_3_1);
  \draw[gray] (p4_3) -- (p4_2_1);
  \draw[gray] (p5_1_1) -- (p3_3_1);
  \draw[gray] (p5_1_1) -- (p4_2_1);
  \draw[gray] (p5_2) -- (p4_3);
  \draw[gray] (p5_2) -- (p5_1_1);
  \draw[gray] (p6_1) -- (p5_2);
  \draw[gray] (p7) -- (p6_1);
\end{tikzpicture}}
\caption{The order $\succeq$ of normalised immanants on $\TNN_n$ for $n=4,5,6,7$. Elements are placed at the height of the longest chain below them.}\label{fig:hasse}
\end{figure}
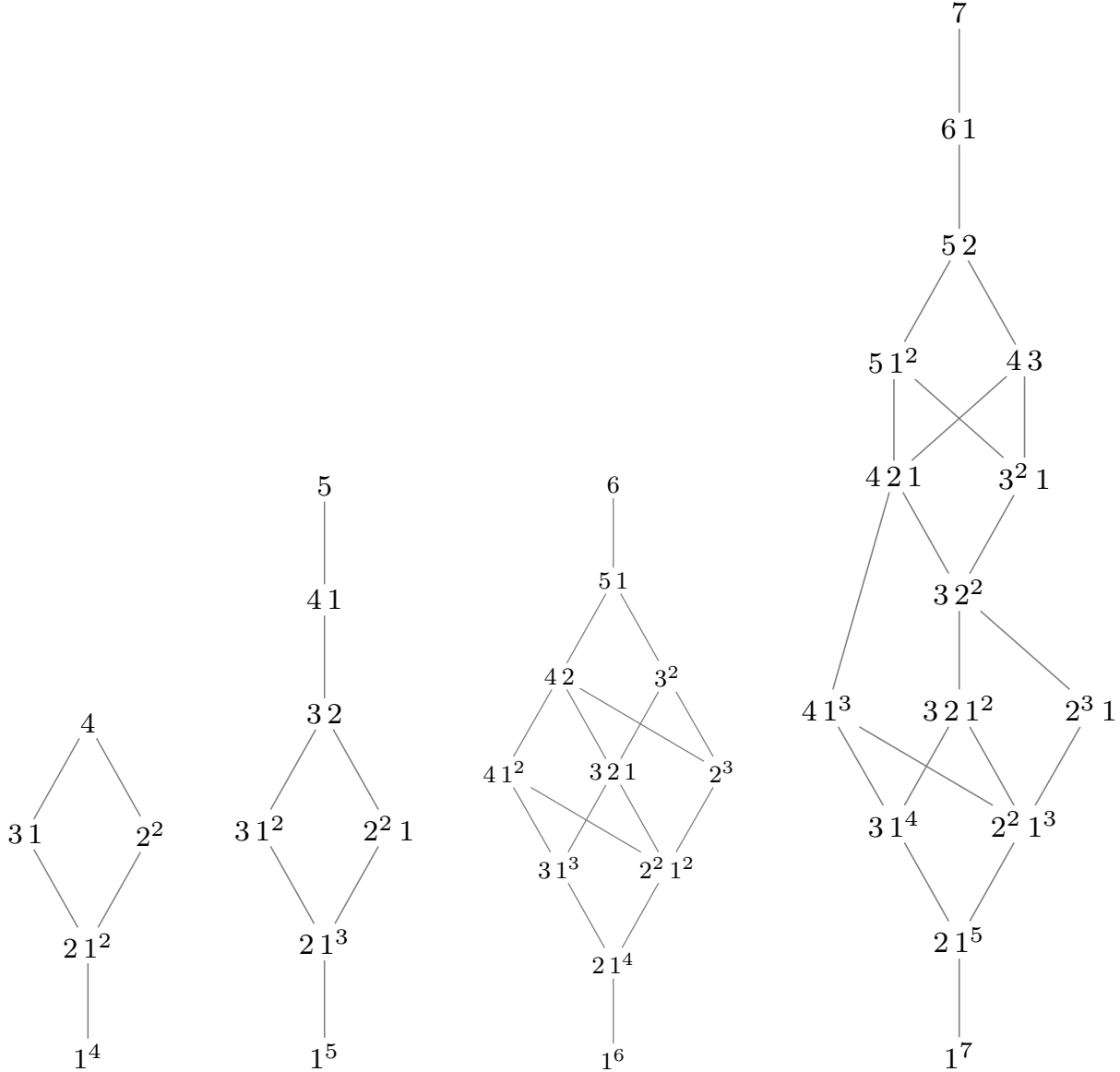

\begin{theorem}\label{thm:structure}
Let $n\le10$.
\begin{enumerate}
\item $\succeq$ equals dominance order iff $n\le3$. The first pair that is comparable in dominance but not in $\succeq$ is $(3,1)$ and $(2,2)$.
\item $\succeq$ is a lattice iff $n\le5$. For $n=6$, the elements $(4,2)$ and $(3,3)$ have two maximal common lower bounds, $(3,2,1)$ and $(2,2,2)$, and hence no meet (Figure~\ref{fig:nolattice}).
\item The hook chain $(n)\succ(n-1,1)\succ\cdots\succ(1^n)$ and the two-column chain $(2^k,1^{n-2k})\succ(2^{k-1},1^{n-2k+2})$ are chains of $\succeq$. The two-row chain $(n-k,k)$, $0\le k\le n/2$, is a chain only for $n\in\{1,2,3,5,7\}$. For even $n\ge4$ it fails at $(n/2+1,n/2-1)\not\succeq(n/2,n/2)$, and for $n=9$ at $(6,3)\not\succeq(5,4)$.
\item Conjugation $a\mapsto a'$ is not an anti-automorphism of $\succeq$ for $n\ge4$.
\end{enumerate}
\end{theorem}

\begin{proof}
Finite computations with Theorem~\ref{thm:order}.
\end{proof}

Theorem~\ref{thm:structure} is a statement about $n\le10$ only, the range of the exhaustive computation; we know of no argument for larger $n$. A failure of a relation $a\succeq b$ is witnessed by some $J_\mu$ and so holds without Theorem~\ref{thm:order}, while every assertion that a relation holds, and hence the lattice property for $n\le5$, uses it.

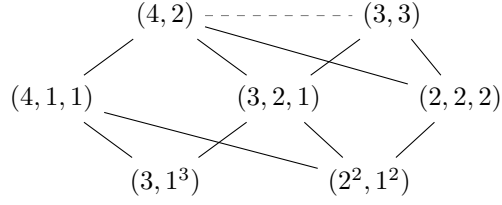
\begin{figure}[ht]
\centering
\begin{tikzpicture}[x=1.5cm,y=1.1cm,every node/.style={font=\small}]
  \node (42) at (0,3) {$(4,2)$};
  \node (33) at (2,3) {$(3,3)$};
  \node (411) at (-1,2) {$(4,1,1)$};
  \node (321) at (1,2) {$(3,2,1)$};
  \node (222) at (2.6,2) {$(2,2,2)$};
  \node (3111) at (0,1) {$(3,1^3)$};
  \node (2211) at (1.8,1) {$(2^2,1^2)$};
  \draw (42) -- (411); \draw (42) -- (321); \draw (42) -- (222);
  \draw (33) -- (321); \draw (33) -- (222);
  \draw (411) -- (3111); \draw (411) -- (2211);
  \draw (321) -- (3111); \draw (321) -- (2211); \draw (222) -- (2211);
  \draw[dashed,gray] (42) -- (33);
\end{tikzpicture}
\caption{Part of the order $\succeq$ for $n=6$ (Theorem~\ref{thm:structure}(2)). Solid lines are covers of $\succeq$ among the seven partitions shown. In dominance $(4,2)$ is above $(3,3)$ (dashed), but in $\succeq$ they are incomparable. Their common lower bounds have two maximal elements, $(3,2,1)$ and $(2,2,2)$, which are incomparable, so $(4,2)\wedge(3,3)$ does not exist. $(4,1,1)$ lies below $(4,2)$ but not below $(3,3)$.}\label{fig:nolattice}
\end{figure}

For every $n$, $(n)$ is the top and $(1^n)$ the bottom element of $\succeq$, and the hook chain is a chain. On $\TNN_n$ the hook chain is due to Skandera \cite[Thm.~27]{Ska25}; the companion paper gives a proof that both statements hold at the level of Kazhdan--Lusztig immanants \cite[Prop.~6.4, Cor.~5.1]{McN}. Table~\ref{tab:structure} gives the size of the order.

\begin{table}[ht]
\small
\begin{tabular}{rrrrrrrrr}
\toprule
$n$ & 3&4&5&6&7&8&9&10\\
\midrule
partitions & 3&5&7&11&15&22&30&42\\
comparable pairs in $\succeq$ & 3&9&20&48&96&198&375&723\\
comparable pairs in dominance & 3&10&21&53&101&216&400&776\\
covers & 2&5&7&16&21&40&55&102\\
height (longest chain) & 3&4&6&7&10&12&15&16\\
width (largest antichain) & 1&2&2&3&3&5&5&7\\
\bottomrule
\end{tabular}
\caption{Size of the normalised order $\succeq$ on partitions of $n$.}\label{tab:structure}
\end{table}

\subsection{Pate's inequalities}
\begin{corollary}\label{cor:pate}
Let $n\le10$, let $\lambda\vdash n$ with $\lambda_1\ge2$, and let $\mu$ be obtained from $\lambda$ by removing its last column and adding its boxes to the bottom of the first column. Then $N_\lambda\ge N_\mu$ on $\TNN_n$, and the difference is KL-positive.
\end{corollary}

\begin{proof}
Pate \cite{Pat91} proved $N_\lambda\ge N_\mu$ on positive semidefinite matrices, so the difference lies in $\CJ$, which equals $\CKL$ for $n\le10$ by Corollary~\ref{cor:cones}.
\end{proof}

\begin{exa}\label{exa:pate}
For $\lambda=(3,2)$ the last column has one box, which moves to the bottom of the first column, giving $\mu=(2,2,1)$:
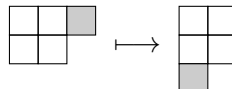
\[
\begin{tikzpicture}[scale=0.38,baseline=(c.base)]
  \foreach \x/\y in {0/0,1/0,0/-1,1/-1} \draw (\x,\y) rectangle ++(1,1);
  \fill[gray!40] (2,0) rectangle ++(1,1); \draw (2,0) rectangle ++(1,1);
  \node (c) at (1.5,-0.6) {};
\end{tikzpicture}
\ \longmapsto\
\begin{tikzpicture}[scale=0.38,baseline=(c.base)]
  \foreach \x/\y in {0/0,1/0,0/-1,1/-1} \draw (\x,\y) rectangle ++(1,1);
  \fill[gray!40] (0,-2) rectangle ++(1,1); \draw (0,-2) rectangle ++(1,1);
  \node (c) at (1,-0.6) {};
\end{tikzpicture}
\]
Here $f^\lambda=f^\mu=5$, and for $\nu=(5),(4,1),(3,2),(3,1,1),(2,2,1),(2,1^3),(1^5)$,
\[
K_{\lambda,\nu}=(0,0,1,1,2,3,5),\qquad K_{\mu,\nu}=(0,0,0,0,1,2,5),
\]
so $K_{\lambda,\nu}/f^\lambda-K_{\mu,\nu}/f^\mu=\frac15(0,0,1,1,1,1,0)\ge0$, as Theorem~\ref{thm:order} requires.
\end{exa}

Independently of Pate's theorem, the Kostka criterion of Theorem~\ref{thm:order} holds for every Pate pair with $n\le12$.

\subsection{An inequality between several normalised immanants}
At the identity matrix every $N_a$ equals $1$, so the permanent cannot dominate a sum of two or more of them. Measured against the determinant, however, it can.

\begin{proposition}\label{prop:excess}
Let $n\le10$ and let $S$ be a set of partitions of $n$ different from $(n)$ and $(1^n)$. Then
\[
\per-\det\ \ge\ \sum_{a\in S}\bigl(N_a-\det\bigr)\quad\text{on }\TNN_n
\qquad\iff\qquad
\sum_{a\in S}\frac{K_{a,\mu}}{f^a}\le1\quad\text{for all }\mu\ne1^n .
\]
\end{proposition}

\begin{proof}
The difference of the two sides is $d_\phi$ with $\phi=\chi^{(n)}+(|S|-1)\chi^{(1^n)}-\sum_{a\in S}\chi^a/f^a$. We have $\langle\phi,\eta^\mu\rangle=1-\sum_SK_{a,\mu}/f^a$ for $\mu\ne1^n$, and $\langle\phi,\eta^{1^n}\rangle=0$. Apply Corollary~\ref{cor:cones}.
\end{proof}

The largest such sets have $1,2,2,3,3,4,4,4$ elements for $n=3,\dots,10$. The smallest example in which total non-negativity is essential occurs at $n=6$.

\begin{exa}\label{ex:six}
On $\TNN_6$,
\begin{equation}\label{eq:six}
\per+2\det\ \ge\ \frac{\Imm_{\chi^{3111}}}{10}+\frac{\Imm_{\chi^{222}}}{5}+\frac{\Imm_{\chi^{21111}}}{5}.
\end{equation}
After multiplying by $10$, the values of $(\text{left}-\text{right})(J_\mu)/\mu!$ are:
\begin{center}
\small
\begin{tabular}{c*{11}{c}}
\toprule
$\mu$ & $6$ & $51$ & $42$ & $411$ & $33$ & $321$ & $31^3$ & $2^3$ & $2^21^2$ & $21^4$ & $1^6$\\
\midrule
value & $10$ & $10$ & $10$ & $10$ & $10$ & $10$ & $9$ & $8$ & $7$ & $0$ & $0$\\
\bottomrule
\end{tabular}
\end{center}
Equality holds at $J_{21^4}$ and at the identity.
\end{exa}

\begin{proposition}\label{prop:notpairwise}
Inequality \eqref{eq:six} fails on some entrywise non-negative matrices, and it is not a non-negative combination of two-term inequalities $\alpha N_a-\beta N_b\ge0$ valid on $\TNN_6$ ($\alpha,\beta\ge0$). This includes $\per\ge N_a$, $N_a\ge0$, and every relation of $\succeq$.
\end{proposition}

\begin{proof}
For the permutation matrix of a $6$-cycle, ten times the difference of the two sides is the value of the class function at a $6$-cycle, which is $-11$. Suppose $T=\per+2\det-N_{3111}-N_{222}-N_{21111}$ is a combination $\sum_i(\alpha_iN_{a_i}-\beta_iN_{b_i})$ of such inequalities. Each term is non-negative at the identity, so $\alpha_i\ge\beta_i$. Since $T(I)=0$, every term has $\alpha_i=\beta_i$. The $N_a$ are linearly independent, so the coefficient $2$ of $N_{1^6}=\det$ must come from a term $\alpha_i(\det-N_b)$ with $b\ne1^6$. But $\det(J_b)=0<N_b(J_b)$, so such a term is not valid on $\TNN_6$.
\end{proof}

\section{Beyond class functions}\label{sec:beyond}

Let $\Pi_n$ be the $n!$-dimensional span of the permutation monomials $x^v=x_{1,v(1)}\cdots x_{n,v(n)}$. Every $F\in\Pi_n$ is uniquely $F=\sum_wa_w\Imm_w$, and we put
\[
\CKL(\Pi_n)=\{F:\ a_w\ge0\ \forall w\},\qquad \CTNN(\Pi_n)=\{F:\ F\ge0\text{ on }\TNN_n\}.
\]
Then $\CKL(\Pi_n)\subseteq\CTNN(\Pi_n)$, and the class-function cones are their intersections with the immanants $d_\phi$. Let $E\subseteq\R^{S_n}$ be the closed convex cone generated by the evaluation vectors $(\Imm_v(A))_{v\in S_n}$, $A\in\TNN_n$. It lies in the non-negative orthant, and $\CTNN(\Pi_n)$ is its dual cone. Call $w$ \emph{TNN-isolated} if the coordinate vector $e_w$ lies in $E$. By separation,
\[
w\ \text{is TNN-isolated}\iff a_w\ge0\ \text{for every }F=\sum_va_v\Imm_v\in\CTNN(\Pi_n),
\]
so $\CKL(\Pi_n)=\CTNN(\Pi_n)$ iff every $w$ is TNN-isolated. These cones go back to Stembridge \cite[\S5]{Ste92}, who studied the cone of total positivity in $\R S_n$, whose dual is $\CTNN(\Pi_n)$, and to Rhoades--Skandera \cite[\S7]{RS06}. Rhoades and Skandera noted $\CKL(\Pi_n)\subsetneq\CTNN(\Pi_n)$ for $n\ge4$ and asked whether every smooth $w$ is TNN-isolated \cite[Question~7.1]{RS06}. Skandera \cite[Thms.~5.4, 5.5]{Ska08} proved that it is, for every $n$: he constructs $A_w\in\TNN_n$ with $\Imm_v(A_w)=\delta_{v,w}$.

\emph{Charts.} A \emph{chart} is a sequence of elementary factors $I+tE_{i+1,i}$ and $I+tE_{i,i+1}$ whose lower letters form a reduced word of $w_0$, as do its upper letters, in any interleaving. By Fomin--Zelevinsky \cite{FZ}, positive parameters give every totally positive matrix up to a positive diagonal factor, which multiplies every element of $\Pi_n$ by its determinant. Since totally positive matrices are dense in $\TNN_n$, an $F$ whose expansion in some chart has only non-negative coefficients is non-negative on $\TNN_n$. In a fixed chart, $\Imm_w=\sum_\alpha\kappa_\alpha(w)t^\alpha$ with $\kappa_\alpha(w)\ge0$. Write $\kappa_\alpha=(\kappa_\alpha(w))_w$, and let $\mathcal N$ be the Newton polytope, the convex hull of the exponents $\alpha$ with $\kappa_\alpha\ne0$.

\begin{theorem}[isolation criterion]\label{thm:isolation}
$w$ is TNN-isolated iff some vertex $\alpha$ of $\mathcal N$ has $\supp\kappa_\alpha=\{w\}$.
\end{theorem}

\begin{proof}
Suppose $\alpha$ is such a vertex, and let $c$ be a linear functional maximised on $\mathcal N$ only at $\alpha$. Put $t=s^c$ and let $s\to\infty$. The evaluation vector, divided by $s^{c\cdot\alpha}$, tends to $\kappa_\alpha$, which is a positive multiple of $e_w$. So $e_w\in E$.

Conversely, suppose $e_w\in E$. Write $\mathrm{ev}(A)=(\Imm_v(A))_{v\in S_n}$ and, when $\mathrm{ev}(A)\ne0$, $u(A)=\mathrm{ev}(A)/\|\mathrm{ev}(A)\|_1$; these normalised vectors generate $E$ as a closed cone. We show in four steps that $e_w$ is a limit of $u(A_k)$ for matrices $A_k$ in the fixed chart, and then read off the vertex.

\emph{Step 1: $e_w$ is a limit of normalised evaluation vectors.} Since $e_w$ spans an extreme ray of the non-negative orthant, which contains $E$, this holds: if $\sum_kc_ku_k\to e_w$ with $c_k\ge0$ and all $u_k\ge0$ of unit $\ell^1$-norm, then $\sum_kc_k\to1$ and $\sum_kc_k(1-u_k(w))\to0$, so some $u_k$ have $u_k(w)\to1$, hence $u_k\to e_w$.

\emph{Step 2: totally positive matrices suffice.} The map $A\mapsto u(A)$ is continuous where $\mathrm{ev}(A)\ne0$, and totally positive matrices are dense in $\TNN_n$. So every $u(A)$, $A\in\TNN_n$, is a limit of $u(A')$ with $A'$ totally positive, and $e_w$ is a limit of $u(A_k)$ with all $A_k$ totally positive.

\emph{Step 3: one fixed chart.} By Fomin--Zelevinsky \cite{FZ}, for the fixed chart every totally positive matrix is $D\,M(t)$, where $M(t)$ is the product of the elementary factors of the chart at positive parameters $t$ and $D$ is a positive diagonal matrix. A diagonal factor in any position can be moved to the left, because $D(I+tE_{ij})=(I+td_id_j^{-1}E_{ij})D$ only rescales a parameter by a positive number. Every $F\in\Pi_n$ satisfies $F(DM)=\det D\cdot F(M)$, so $u(D\,M(t))=u(M(t))=\kappa\tau(t)/\|\kappa\tau(t)\|_1$, where $\mathcal A$ is the set of exponents $\alpha$ with $\kappa_\alpha\ne0$, $\kappa$ is the matrix with columns $\kappa_\alpha$, and $\tau(t)=(t^\alpha)_{\alpha\in\mathcal A}/\|(t^\alpha)_\alpha\|_1$.

\emph{Step 4: the toric closure.} As $t$ ranges over the positive orthant, $\tau(t)$ ranges over the positive part of the projective toric variety of $\mathcal A$. Its closure is the disjoint union, over the faces $\mathcal F$ of $\mathcal N$, of the positive parts of the torus orbits, whose points $y$ satisfy $y_\alpha>0$ for $\alpha\in\mathcal F$ and $y_\alpha=0$ otherwise \cite{Sot}. On this compact set the map $y\mapsto\kappa y/\|\kappa y\|_1$ is continuous, because $\kappa\ge0$ has no zero column, so $\|\kappa y\|_1\ge\min_\alpha\|\kappa_\alpha\|_1>0$ for $y\ge0$ with $\|y\|_1=1$. So the limit $e_w$ equals $\kappa y/\|\kappa y\|_1$ for some $y$ in the closure, supported exactly on some face $\mathcal F$, that is, $e_w$ is a positive multiple of $\sum_{\alpha\in\mathcal F}\kappa_\alpha y_\alpha$ with all $y_\alpha>0$. Since all $\kappa_\alpha\ge0$, every $\kappa_\alpha$ with $\alpha\in\mathcal F$ is supported on $\{w\}$. In particular this holds for a vertex of $\mathcal F$, which is a vertex of $\mathcal N$.
\end{proof}

\begin{theorem}\label{thm:smalln}
\begin{enumerate}
\item For $n\le3$ every $w\in S_n$ is TNN-isolated, so $\CKL(\Pi_n)=\CTNN(\Pi_n)$.
\item For $n=4$ and $n=5$, $w$ is TNN-isolated iff $w$ avoids $3412$ and $4231$. These are $22$ of the $24$ elements of $S_4$ and $88$ of the $120$ elements of $S_5$. In particular $\CKL(\Pi_n)\subsetneq\CTNN(\Pi_n)$ for $n=4,5$.
\end{enumerate}
\end{theorem}

Both directions of part (2) are proved here directly from Theorem~\ref{thm:isolation}, without using \cite{Ska08} or \cite{McN}. For every $n$, ``smooth $\Rightarrow$ isolated'' is due to Skandera \cite{Ska08}, and the companion paper gives a proof of the converse by a different method \cite[Thm.~9.6]{McN}.

\begin{proof}
Apply Theorem~\ref{thm:isolation} to the chart whose lower and upper words are both $s_1\,s_2s_1\,s_3s_2s_1\allowbreak\cdots\allowbreak s_{n-1}\cdots s_1$, all lower factors first, by computing the vertices of $\mathcal N$ that carry a single $\kappa$-support. All certificates are exact and archived (\nolinkurl{results/tnn_isolation_certificates_n4.json} and \nolinkurl{_n5.json}). For each isolated $w$ there is an exponent $\alpha$ with $\supp\kappa_\alpha=\{w\}$ and an integer functional maximised on $\mathcal N$ only at $\alpha$. For each singular $w$, every exponent $\alpha$ with $\supp\kappa_\alpha=\{w\}$ is the midpoint of two other exponents, hence not a vertex ($1$ and $18$ such exponents for $3412$ and $4231$; $4{,}747$ in all for $n=5$, and some singular $w$ have none). Both kinds of certificate are re-checked in integer arithmetic by \texttt{check\_isolation\_certificates.py --verify}.
\end{proof}

\begin{exa}\label{exa:isolation}
For $n=3$ the chart of the proof is $L_1(t_1)L_2(t_2)L_1(t_3)\,U_1(u_1)U_2(u_2)U_1(u_3)$, with $L_i(t)=I+tE_{i+1,i}$ and $U_i(t)=I+tE_{i,i+1}$. The six $\Imm_w$ have $14$ distinct exponents in total. For each $w$ the table gives a monomial $t^\alpha$ that occurs in $\Imm_w$ with coefficient $1$ and in no other $\Imm_v$, and a functional $c$ with $c\cdot\alpha>c\cdot\beta$ for the other $13$ exponents $\beta$. So $\alpha$ is a vertex of $\mathcal N$ with $\kappa_\alpha=e_w$, and Theorem~\ref{thm:isolation} shows that every $w\in S_3$ is TNN-isolated.
\begin{center}
\small
\begin{tabular}{ccc}
\toprule
$w$ & $t^\alpha$ & $c$ on $(t_1,t_2,t_3,u_1,u_2,u_3)$\\
\midrule
$123$ & $1$ & $(-1,-1,-1,-1,-1,-1)$\\
$132$ & $t_2u_2$ & $(-1,1,-1,-1,1,-1)$\\
$213$ & $t_3u_3$ & $(-1,-1,1,-1,-1,1)$\\
$231$ & $t_2t_3u_2u_3$ & $(-1,1,0,-1,1,1)$\\
$312$ & $t_1t_2u_1u_2$ & $(1,1,-1,0,1,-1)$\\
$321$ & $t_2t_3^2u_1u_2u_3$ & $(-1,1,1,0,1,1)$\\
\bottomrule
\end{tabular}
\end{center}
\end{exa}

Explicit elements of $\CTNN(\Pi_5)\setminus\CKL(\Pi_5)$ are given in \cite[Thm.~9.2]{McN}, for example $\Imm_{21543}+\Imm_{54321}-\Imm_{52431}$. For $n=4$, Skandera and Zelevinsky give $\Imm_{3214}+\Imm_{1432}-\Imm_{3412}\in\CTNN(\Pi_4)\setminus\CKL(\Pi_4)$ (reported in \cite[(5.1)]{Ska08}), certified by a subtraction-free expression in minors. The stronger element $\tfrac12(\Imm_{1432}+\Imm_{3214})-\Imm_{3412}$ is non-negative on $\TNN_4$ by \cite[Prop.~9.3]{McN}. Lifting it, and the elements $\varepsilon\Imm_{2143}+\tfrac3{4\varepsilon}\Imm_{4321}-\Imm_{4231}$ for small $\varepsilon>0$, along an occurrence of $3412$ or $4231$ gives a proof for every $n$ that every TNN-isolated $w$ is smooth \cite[Thm.~9.6]{McN}, the converse of Skandera's theorem. Theorem~\ref{thm:smalln} confirms this for $n\le5$ by a different method.

\section{Further computations}\label{sec:evidence}

This section records computations about results that are proved, or conjectured, elsewhere, mostly in the companion preprint \cite{McN}. They are cross-checks of those proofs and evidence for those conjectures, not replacements for them, and nothing in Sections~\ref{sec:intro}--\ref{sec:beyond} depends on them.

\subsection{Clean steps}
A step $w\to ws$ with $sw<w<ws$ is \emph{clean} if no $z<w$ with $zs<z$ has $\mu(z,w)\ne0$; then $c_{\cdot,ws}=2c_{\cdot,w}$ \cite[\S7.2]{McN}. The steps certified by the bad-cover criterion of \cite[Prop.~7.8]{McN}, starting from smooth permutations, reach $84\%$, $68\%$, $49\%$, $32\%$ and $19\%$ of the singular permutations for $n=5,\dots,9$; all clean steps together reach $84\%$, $68\%$, $51\%$, $35\%$ and $22\%$. The singular $w$ that are not the top of a clean step are a stable $30\%$ of all singular $w$ for $n=7,8,9$.

\subsection{Checks of the results for all $n$}
The results of \cite{McN} were also checked by computer. The product theorem \cite[Thm.~3.1]{McN} holds for all $I,T,u,v$ with $n\le6$, and its single-entry case for $n\le7$. The identity $\ch D_k=\sum_jh_jm_{(k-j,1^{n-k})}$ of Skandera's hook chain holds for $n\le12$, and $D_k$ and every term have non-negative Kazhdan--Lusztig coefficients for $n\le8$. For the hook theorem, the closure identity was checked for $m\le7$, the formulas for the rotations and for $H_c(e)$ for $m\le9$, the boundary signs for $m\le10$, and the whole recursion numerically for $m\le7$. The diagram structure in the bottom-terms lemma was checked exactly from the Kazhdan--Lusztig polynomials for $m\le6$, without the colouring rule.

For the isolation theorem \cite[Thm.~9.6]{McN} we used the exact values $P_{\sigma,w}(1)$ from the extremal-pair tables. For each singular $w$ with $n\le10$ (there are $3{,}504{,}490$ for $n=10$), take the first occurrence of $3412$ in lexicographic order, or of $4231$ if there is none. Lift $\Imm_{1432}+\Imm_{3214}-2\Imm_{3412}$, respectively $\Imm_{2143}+\Imm_{4321}-\Imm_{4231}$, along it as in the proof. In every case the coefficient of $\Imm_w$ is negative. The proof itself needs no choice of occurrence. Its smoothness lemma \cite[Lemma~9.5]{McN} says that $X_w$ is smooth at $w(i\,j)$ iff the entries of $w$ inside the rectangle spanned by positions $i,j$ decrease, and that then $P_{w(i\,j),w}=1$. We checked that $P_{w(i\,j),w}=1$ iff these entries decrease, for every inversion of every $w$ with $n\le10$ ($81{,}648{,}000$ inversions for $n=10$).

\subsection{Hook expansions}
Let $B_{ij}(w)$ be the coefficient of $\Imm_w$ in $x_{ij}x_{ji}\det x[\{i,j\}^c,\{i,j\}^c]$. The conjectural closed form for $B_{ij}$ in \cite[Conj.~10.1]{McN} matches the computation for every pair $(i,j)$ with $n\le8$, and the inertness property stated after it holds for every $2$- and $3$-cycle with $n\le8$. For $a\ge2$, $f^{(a)}_i(\Cp w)=0$ unless $w$ moves a single block of consecutive letters containing $i$, for $n\le8$. The functions $g_T$ in the decomposition of Skandera's $\theta_\ell$ are KL-positive for $n\le8$ (\nolinkurl{results/theta_per_set_n8.log}).

\subsection{The shapes $(3,2^b,1^c)$}
For these shapes Conjecture~\ref{conj:H} reduces to $321$-avoiding $w$ containing both $312$ and $231$ \cite[Props.~7.2, 8.2]{McN}. There are $40$, $130$ and $427$ indecomposable such $w$ for $n=6,7,8$, for example for $(3,2,1)$, $(3,2,1,1)$ and $(3,2,2,1)$, and Theorem~\ref{thm:main} covers $3+2b+c\le10$.

\subsection*{Open problems}
\begin{enumerate}
\item \emph{Scalable and certifiable computation.} The extremal-pair method applies unchanged to $n=11$, but $S_{11}$ has an estimated $5\times10^{10}$ extremal pairs; extrapolating from $n=9,10$, the graded step needs a few times $10^4$ core-hours and several terabytes of working memory. Find an algorithm that is sharded by descent classes and streamed to disk, and that emits certificates which can be checked independently and more cheaply than the computation itself, so that a result for $n=11$ would not rest on a single run.
\item \emph{The Kostka order.} The criterion $K_{a,\mu}/f^a\ge K_{b,\mu}/f^b$ for all $\mu\vdash n$ of Theorem~\ref{thm:order} defines a partial order on partitions of every $n$, contained in dominance order; it agrees with $\succeq$ on $\TNN_n$ for $n\le10$, and its failures are failures of $\succeq$ for every $n$. Describe this order combinatorially: which pairs comparable in dominance are incomparable in it, and what obstructs the lattice property (Theorem~\ref{thm:structure}).
\item \emph{Chart certificates.} For each singular $w$, an element of $\CTNN(\Pi_n)$ with a negative coefficient at $\Imm_w$ and non-negative coefficients elsewhere exists by the proof given in \cite[Thm.~9.6]{McN}. Find such elements whose expansion in a single chart of Theorem~\ref{thm:isolation} is subtraction-free. Our search found them for only $3$ of the $32$ singular $w\in S_5$ in $200$ random charts, and for neither singular $w\in S_4$ in $30$ charts (\nolinkurl{results/tnn_isolation_n4.json}, \nolinkurl{_n5.json}). Decide whether they exist for $3412$ and $4231$, and whether they can be constructed uniformly, for example by lifting along pattern occurrences. Relatedly, find compact, human-readable Newton-polytope certificates for Theorem~\ref{thm:smalln}, and use them to give an independent check of the isolation theorem for $n=6$.
\item \emph{The graded cones.} For $n\le10$, $\CJ(q)\subseteq\CKL(q)\subseteq\CJ^{\mathrm{ev}}(q)$, the first inclusion strict for $n=2$ and the second for every $n\ge3$ (Section~\ref{sec:consequences}). Describe $\CKL(q)$, for example by finding its generators as an $\N[q]$-semimodule for $n=3,4$, and determine which traces in $\CJ^{\mathrm{ev}}(q)\setminus\CJ(q)$ lie in it.
\end{enumerate}

\clearpage
\appendix

\section{Hasse diagrams for \texorpdfstring{$n=8,9,10$}{n=8,9,10}}\label{app:hasse}
\raggedbottom

\begin{figure}[!h]
\centering
\resizebox{0.44\textwidth}{!}{\begin{tikzpicture}[every node/.style={font=\scriptsize,inner sep=1.5pt}]
  \node (p8) at (0.00,12.10) {$8$};
  \node (p7_1) at (0.00,11.00) {$7\,1$};
  \node (p6_2) at (0.00,9.90) {$6\,2$};
  \node (p6_1_1) at (-0.62,8.80) {$6\,1^{2}$};
  \node (p5_3) at (0.62,8.80) {$5\,3$};
  \node (p5_2_1) at (-0.62,7.70) {$5\,2\,1$};
  \node (p5_1_1_1) at (-2.50,4.40) {$5\,1^{3}$};
  \node (p4_4) at (0.62,7.70) {$4^{2}$};
  \node (p4_3_1) at (0.00,6.60) {$4\,3\,1$};
  \node (p4_2_2) at (-0.62,5.50) {$4\,2^{2}$};
  \node (p4_2_1_1) at (-1.25,4.40) {$4\,2\,1^{2}$};
  \node (p4_1_1_1_1) at (-1.25,3.30) {$4\,1^{4}$};
  \node (p3_3_2) at (0.62,5.50) {$3^{2}\,2$};
  \node (p3_3_1_1) at (0.00,4.40) {$3^{2}\,1^{2}$};
  \node (p3_2_2_1) at (1.25,4.40) {$3\,2^{2}\,1$};
  \node (p3_2_1_1_1) at (0.00,3.30) {$3\,2\,1^{3}$};
  \node (p3_1_1_1_1_1) at (-0.62,2.20) {$3\,1^{5}$};
  \node (p2_2_2_2) at (2.50,4.40) {$2^{4}$};
  \node (p2_2_2_1_1) at (1.25,3.30) {$2^{3}\,1^{2}$};
  \node (p2_2_1_1_1_1) at (0.62,2.20) {$2^{2}\,1^{4}$};
  \node (p2_1_1_1_1_1_1) at (0.00,1.10) {$2\,1^{6}$};
  \node (p1_1_1_1_1_1_1_1) at (0.00,0.00) {$1^{8}$};
  \draw[gray] (p2_1_1_1_1_1_1) -- (p1_1_1_1_1_1_1_1);
  \draw[gray] (p2_2_1_1_1_1) -- (p2_1_1_1_1_1_1);
  \draw[gray] (p2_2_2_1_1) -- (p2_2_1_1_1_1);
  \draw[gray] (p2_2_2_2) -- (p2_2_2_1_1);
  \draw[gray] (p3_1_1_1_1_1) -- (p2_1_1_1_1_1_1);
  \draw[gray] (p3_2_1_1_1) -- (p2_2_1_1_1_1);
  \draw[gray] (p3_2_1_1_1) -- (p3_1_1_1_1_1);
  \draw[gray] (p3_2_2_1) -- (p2_2_2_1_1);
  \draw[gray] (p3_2_2_1) -- (p3_2_1_1_1);
  \draw[gray] (p3_3_1_1) -- (p2_2_2_1_1);
  \draw[gray] (p3_3_1_1) -- (p3_2_1_1_1);
  \draw[gray] (p3_3_2) -- (p2_2_2_2);
  \draw[gray] (p3_3_2) -- (p3_2_2_1);
  \draw[gray] (p3_3_2) -- (p3_3_1_1);
  \draw[gray] (p4_1_1_1_1) -- (p2_2_1_1_1_1);
  \draw[gray] (p4_1_1_1_1) -- (p3_1_1_1_1_1);
  \draw[gray] (p4_2_1_1) -- (p2_2_2_1_1);
  \draw[gray] (p4_2_1_1) -- (p3_2_1_1_1);
  \draw[gray] (p4_2_1_1) -- (p4_1_1_1_1);
  \draw[gray] (p4_2_2) -- (p2_2_2_2);
  \draw[gray] (p4_2_2) -- (p3_2_2_1);
  \draw[gray] (p4_2_2) -- (p3_3_1_1);
  \draw[gray] (p4_2_2) -- (p4_2_1_1);
  \draw[gray] (p4_3_1) -- (p3_3_2);
  \draw[gray] (p4_3_1) -- (p4_2_1_1);
  \draw[gray] (p4_4) -- (p4_2_2);
  \draw[gray] (p4_4) -- (p4_3_1);
  \draw[gray] (p5_1_1_1) -- (p2_2_2_1_1);
  \draw[gray] (p5_1_1_1) -- (p3_2_1_1_1);
  \draw[gray] (p5_1_1_1) -- (p4_1_1_1_1);
  \draw[gray] (p5_2_1) -- (p4_2_2);
  \draw[gray] (p5_2_1) -- (p4_3_1);
  \draw[gray] (p5_2_1) -- (p5_1_1_1);
  \draw[gray] (p5_3) -- (p5_2_1);
  \draw[gray] (p6_1_1) -- (p5_2_1);
  \draw[gray] (p6_2) -- (p5_3);
  \draw[gray] (p6_2) -- (p6_1_1);
  \draw[gray] (p7_1) -- (p4_4);
  \draw[gray] (p7_1) -- (p6_2);
  \draw[gray] (p8) -- (p7_1);
\end{tikzpicture}}\hfill
\resizebox{!}{0.78\textheight}{\begin{tikzpicture}[every node/.style={font=\scriptsize,inner sep=1.5pt}]
  \node (p9) at (0.00,15.40) {$9$};
  \node (p8_1) at (0.00,14.30) {$8\,1$};
  \node (p7_2) at (0.00,13.20) {$7\,2$};
  \node (p7_1_1) at (-0.62,12.10) {$7\,1^{2}$};
  \node (p6_3) at (0.62,12.10) {$6\,3$};
  \node (p6_2_1) at (-0.62,11.00) {$6\,2\,1$};
  \node (p6_1_1_1) at (-1.88,9.90) {$6\,1^{3}$};
  \node (p5_4) at (0.62,11.00) {$5\,4$};
  \node (p5_3_1) at (-0.62,9.90) {$5\,3\,1$};
  \node (p5_2_2) at (0.62,9.90) {$5\,2^{2}$};
  \node (p5_2_1_1) at (-0.62,8.80) {$5\,2\,1^{2}$};
  \node (p5_1_1_1_1) at (-2.50,4.40) {$5\,1^{4}$};
  \node (p4_4_1) at (1.88,9.90) {$4^{2}\,1$};
  \node (p4_3_2) at (0.62,8.80) {$4\,3\,2$};
  \node (p4_3_1_1) at (-0.62,7.70) {$4\,3\,1^{2}$};
  \node (p4_2_2_1) at (-0.62,6.60) {$4\,2^{2}\,1$};
  \node (p4_2_1_1_1) at (-1.25,4.40) {$4\,2\,1^{3}$};
  \node (p4_1_1_1_1_1) at (-1.25,3.30) {$4\,1^{5}$};
  \node (p3_3_3) at (0.62,7.70) {$3^{3}$};
  \node (p3_3_2_1) at (0.62,6.60) {$3^{2}\,2\,1$};
  \node (p3_3_1_1_1) at (0.00,4.40) {$3^{2}\,1^{3}$};
  \node (p3_2_2_2) at (0.00,5.50) {$3\,2^{3}$};
  \node (p3_2_2_1_1) at (1.25,4.40) {$3\,2^{2}\,1^{2}$};
  \node (p3_2_1_1_1_1) at (0.00,3.30) {$3\,2\,1^{4}$};
  \node (p3_1_1_1_1_1_1) at (-0.62,2.20) {$3\,1^{6}$};
  \node (p2_2_2_2_1) at (2.50,4.40) {$2^{4}\,1$};
  \node (p2_2_2_1_1_1) at (1.25,3.30) {$2^{3}\,1^{3}$};
  \node (p2_2_1_1_1_1_1) at (0.62,2.20) {$2^{2}\,1^{5}$};
  \node (p2_1_1_1_1_1_1_1) at (0.00,1.10) {$2\,1^{7}$};
  \node (p1_1_1_1_1_1_1_1_1) at (0.00,0.00) {$1^{9}$};
  \draw[gray] (p2_1_1_1_1_1_1_1) -- (p1_1_1_1_1_1_1_1_1);
  \draw[gray] (p2_2_1_1_1_1_1) -- (p2_1_1_1_1_1_1_1);
  \draw[gray] (p2_2_2_1_1_1) -- (p2_2_1_1_1_1_1);
  \draw[gray] (p2_2_2_2_1) -- (p2_2_2_1_1_1);
  \draw[gray] (p3_1_1_1_1_1_1) -- (p2_1_1_1_1_1_1_1);
  \draw[gray] (p3_2_1_1_1_1) -- (p2_2_1_1_1_1_1);
  \draw[gray] (p3_2_1_1_1_1) -- (p3_1_1_1_1_1_1);
  \draw[gray] (p3_2_2_1_1) -- (p2_2_2_1_1_1);
  \draw[gray] (p3_2_2_1_1) -- (p3_2_1_1_1_1);
  \draw[gray] (p3_2_2_2) -- (p2_2_2_2_1);
  \draw[gray] (p3_2_2_2) -- (p3_2_2_1_1);
  \draw[gray] (p3_3_1_1_1) -- (p2_2_2_1_1_1);
  \draw[gray] (p3_3_1_1_1) -- (p3_2_1_1_1_1);
  \draw[gray] (p3_3_2_1) -- (p3_2_2_2);
  \draw[gray] (p3_3_2_1) -- (p3_3_1_1_1);
  \draw[gray] (p3_3_3) -- (p3_3_2_1);
  \draw[gray] (p4_1_1_1_1_1) -- (p2_2_1_1_1_1_1);
  \draw[gray] (p4_1_1_1_1_1) -- (p3_1_1_1_1_1_1);
  \draw[gray] (p4_2_1_1_1) -- (p2_2_2_1_1_1);
  \draw[gray] (p4_2_1_1_1) -- (p3_2_1_1_1_1);
  \draw[gray] (p4_2_1_1_1) -- (p4_1_1_1_1_1);
  \draw[gray] (p4_2_2_1) -- (p3_2_2_2);
  \draw[gray] (p4_2_2_1) -- (p3_3_1_1_1);
  \draw[gray] (p4_2_2_1) -- (p4_2_1_1_1);
  \draw[gray] (p4_3_1_1) -- (p3_3_2_1);
  \draw[gray] (p4_3_1_1) -- (p4_2_2_1);
  \draw[gray] (p4_3_2) -- (p4_3_1_1);
  \draw[gray] (p4_4_1) -- (p4_3_2);
  \draw[gray] (p5_1_1_1_1) -- (p2_2_2_1_1_1);
  \draw[gray] (p5_1_1_1_1) -- (p3_2_1_1_1_1);
  \draw[gray] (p5_1_1_1_1) -- (p4_1_1_1_1_1);
  \draw[gray] (p5_2_1_1) -- (p4_3_1_1);
  \draw[gray] (p5_2_1_1) -- (p5_1_1_1_1);
  \draw[gray] (p5_2_2) -- (p5_2_1_1);
  \draw[gray] (p5_3_1) -- (p4_3_2);
  \draw[gray] (p5_3_1) -- (p5_2_1_1);
  \draw[gray] (p5_4) -- (p3_3_3);
  \draw[gray] (p5_4) -- (p4_4_1);
  \draw[gray] (p5_4) -- (p5_2_2);
  \draw[gray] (p5_4) -- (p5_3_1);
  \draw[gray] (p6_1_1_1) -- (p5_2_1_1);
  \draw[gray] (p6_2_1) -- (p5_2_2);
  \draw[gray] (p6_2_1) -- (p5_3_1);
  \draw[gray] (p6_2_1) -- (p6_1_1_1);
  \draw[gray] (p6_3) -- (p3_3_3);
  \draw[gray] (p6_3) -- (p4_4_1);
  \draw[gray] (p6_3) -- (p6_2_1);
  \draw[gray] (p7_1_1) -- (p3_3_3);
  \draw[gray] (p7_1_1) -- (p4_4_1);
  \draw[gray] (p7_1_1) -- (p6_2_1);
  \draw[gray] (p7_2) -- (p5_4);
  \draw[gray] (p7_2) -- (p6_3);
  \draw[gray] (p7_2) -- (p7_1_1);
  \draw[gray] (p8_1) -- (p7_2);
  \draw[gray] (p9) -- (p8_1);
\end{tikzpicture}}
\caption{The order $\succeq$ of normalised immanants on $\TNN_8$ (left) and $\TNN_9$ (right).}\label{fig:hasse89}
\end{figure}
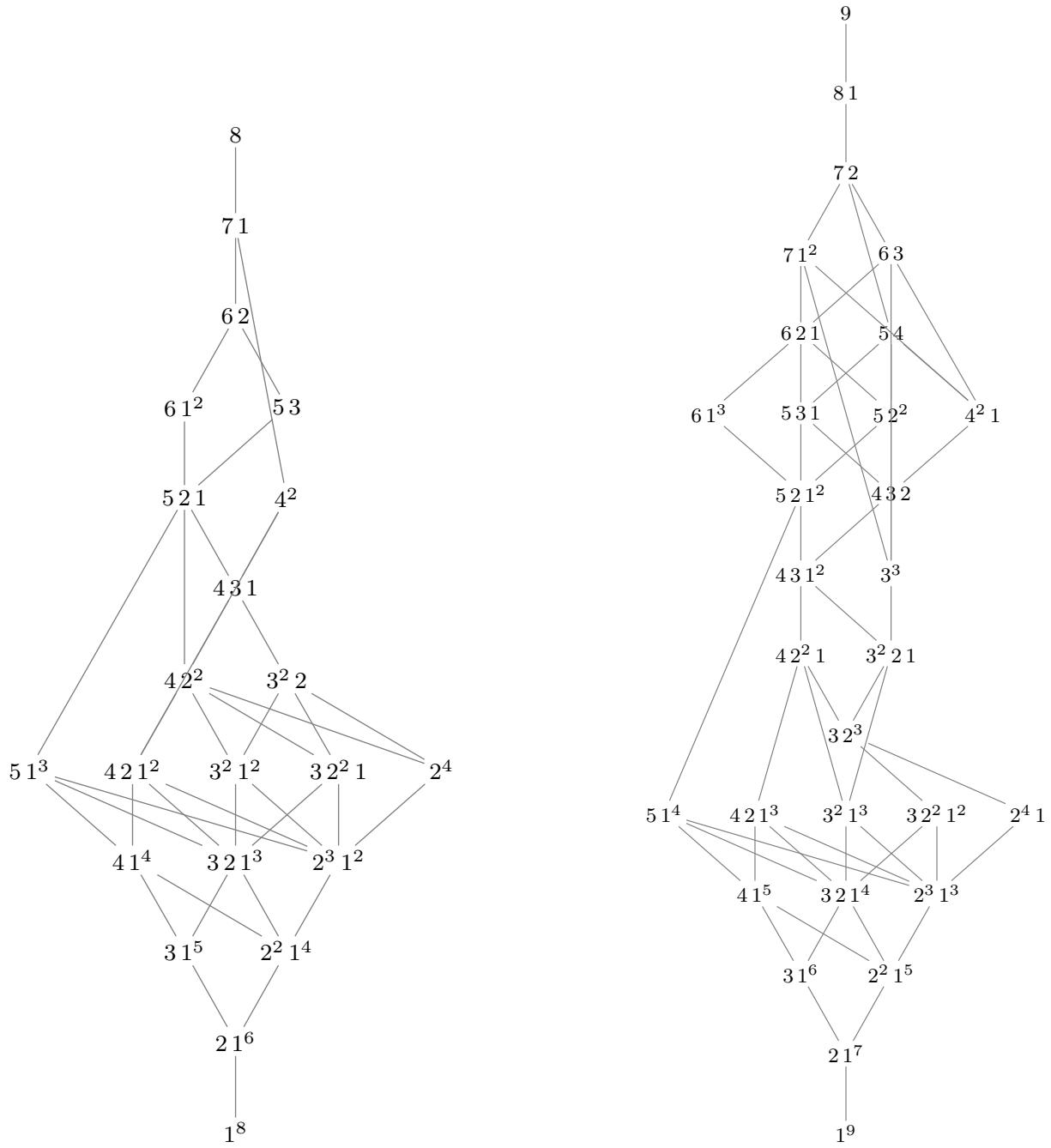

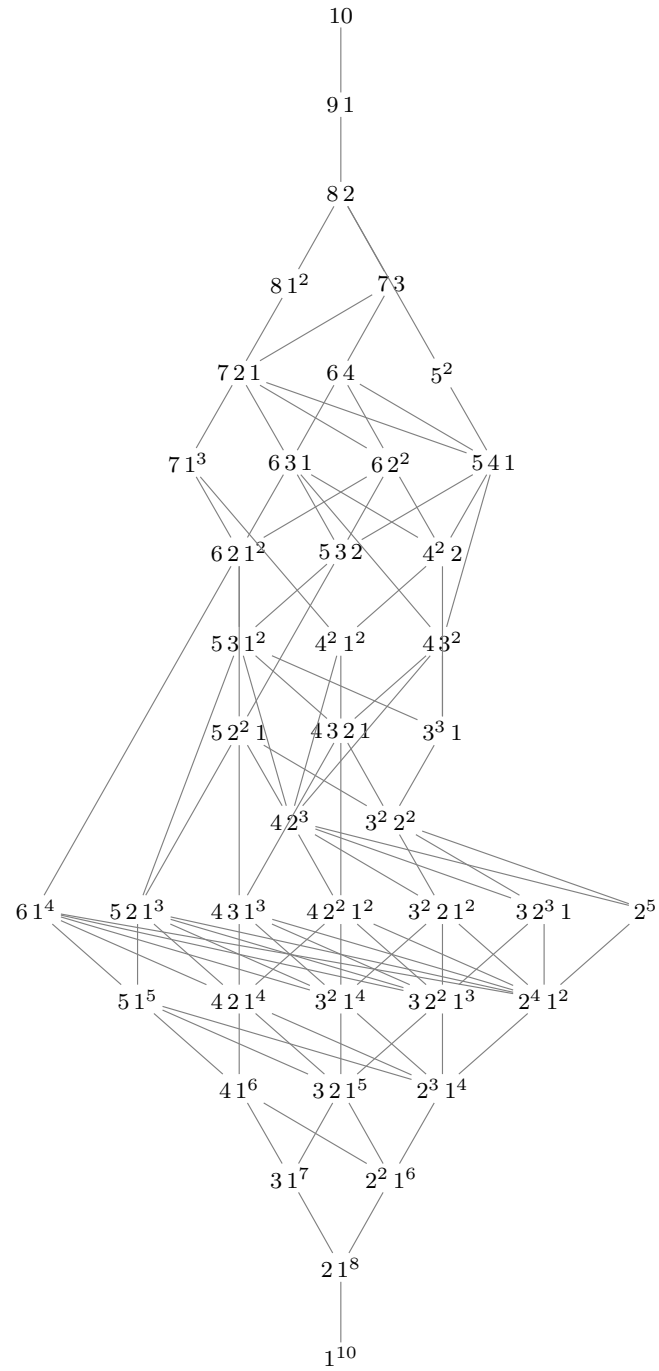
\begin{figure}[!ht]
\centering
\resizebox{!}{0.8\textheight}{\begin{tikzpicture}[every node/.style={font=\scriptsize,inner sep=1.5pt}]
  \node (p10) at (0.00,16.50) {$10$};
  \node (p9_1) at (0.00,15.40) {$9\,1$};
  \node (p8_2) at (0.00,14.30) {$8\,2$};
  \node (p8_1_1) at (-0.62,13.20) {$8\,1^{2}$};
  \node (p7_3) at (0.62,13.20) {$7\,3$};
  \node (p7_2_1) at (-1.25,12.10) {$7\,2\,1$};
  \node (p7_1_1_1) at (-1.88,11.00) {$7\,1^{3}$};
  \node (p6_4) at (0.00,12.10) {$6\,4$};
  \node (p6_3_1) at (-0.62,11.00) {$6\,3\,1$};
  \node (p6_2_2) at (0.62,11.00) {$6\,2^{2}$};
  \node (p6_2_1_1) at (-1.25,9.90) {$6\,2\,1^{2}$};
  \node (p6_1_1_1_1) at (-3.75,5.50) {$6\,1^{4}$};
  \node (p5_5) at (1.25,12.10) {$5^{2}$};
  \node (p5_4_1) at (1.88,11.00) {$5\,4\,1$};
  \node (p5_3_2) at (0.00,9.90) {$5\,3\,2$};
  \node (p5_3_1_1) at (-1.25,8.80) {$5\,3\,1^{2}$};
  \node (p5_2_2_1) at (-1.25,7.70) {$5\,2^{2}\,1$};
  \node (p5_2_1_1_1) at (-2.50,5.50) {$5\,2\,1^{3}$};
  \node (p5_1_1_1_1_1) at (-2.50,4.40) {$5\,1^{5}$};
  \node (p4_4_2) at (1.25,9.90) {$4^{2}\,2$};
  \node (p4_4_1_1) at (0.00,8.80) {$4^{2}\,1^{2}$};
  \node (p4_3_3) at (1.25,8.80) {$4\,3^{2}$};
  \node (p4_3_2_1) at (0.00,7.70) {$4\,3\,2\,1$};
  \node (p4_3_1_1_1) at (-1.25,5.50) {$4\,3\,1^{3}$};
  \node (p4_2_2_2) at (-0.62,6.60) {$4\,2^{3}$};
  \node (p4_2_2_1_1) at (0.00,5.50) {$4\,2^{2}\,1^{2}$};
  \node (p4_2_1_1_1_1) at (-1.25,4.40) {$4\,2\,1^{4}$};
  \node (p4_1_1_1_1_1_1) at (-1.25,3.30) {$4\,1^{6}$};
  \node (p3_3_3_1) at (1.25,7.70) {$3^{3}\,1$};
  \node (p3_3_2_2) at (0.62,6.60) {$3^{2}\,2^{2}$};
  \node (p3_3_2_1_1) at (1.25,5.50) {$3^{2}\,2\,1^{2}$};
  \node (p3_3_1_1_1_1) at (0.00,4.40) {$3^{2}\,1^{4}$};
  \node (p3_2_2_2_1) at (2.50,5.50) {$3\,2^{3}\,1$};
  \node (p3_2_2_1_1_1) at (1.25,4.40) {$3\,2^{2}\,1^{3}$};
  \node (p3_2_1_1_1_1_1) at (0.00,3.30) {$3\,2\,1^{5}$};
  \node (p3_1_1_1_1_1_1_1) at (-0.62,2.20) {$3\,1^{7}$};
  \node (p2_2_2_2_2) at (3.75,5.50) {$2^{5}$};
  \node (p2_2_2_2_1_1) at (2.50,4.40) {$2^{4}\,1^{2}$};
  \node (p2_2_2_1_1_1_1) at (1.25,3.30) {$2^{3}\,1^{4}$};
  \node (p2_2_1_1_1_1_1_1) at (0.62,2.20) {$2^{2}\,1^{6}$};
  \node (p2_1_1_1_1_1_1_1_1) at (0.00,1.10) {$2\,1^{8}$};
  \node (p1_1_1_1_1_1_1_1_1_1) at (0.00,0.00) {$1^{10}$};
  \draw[gray] (p2_1_1_1_1_1_1_1_1) -- (p1_1_1_1_1_1_1_1_1_1);
  \draw[gray] (p2_2_1_1_1_1_1_1) -- (p2_1_1_1_1_1_1_1_1);
  \draw[gray] (p2_2_2_1_1_1_1) -- (p2_2_1_1_1_1_1_1);
  \draw[gray] (p2_2_2_2_1_1) -- (p2_2_2_1_1_1_1);
  \draw[gray] (p2_2_2_2_2) -- (p2_2_2_2_1_1);
  \draw[gray] (p3_1_1_1_1_1_1_1) -- (p2_1_1_1_1_1_1_1_1);
  \draw[gray] (p3_2_1_1_1_1_1) -- (p2_2_1_1_1_1_1_1);
  \draw[gray] (p3_2_1_1_1_1_1) -- (p3_1_1_1_1_1_1_1);
  \draw[gray] (p3_2_2_1_1_1) -- (p2_2_2_1_1_1_1);
  \draw[gray] (p3_2_2_1_1_1) -- (p3_2_1_1_1_1_1);
  \draw[gray] (p3_2_2_2_1) -- (p2_2_2_2_1_1);
  \draw[gray] (p3_2_2_2_1) -- (p3_2_2_1_1_1);
  \draw[gray] (p3_3_1_1_1_1) -- (p2_2_2_1_1_1_1);
  \draw[gray] (p3_3_1_1_1_1) -- (p3_2_1_1_1_1_1);
  \draw[gray] (p3_3_2_1_1) -- (p2_2_2_2_1_1);
  \draw[gray] (p3_3_2_1_1) -- (p3_2_2_1_1_1);
  \draw[gray] (p3_3_2_1_1) -- (p3_3_1_1_1_1);
  \draw[gray] (p3_3_2_2) -- (p2_2_2_2_2);
  \draw[gray] (p3_3_2_2) -- (p3_2_2_2_1);
  \draw[gray] (p3_3_2_2) -- (p3_3_2_1_1);
  \draw[gray] (p3_3_3_1) -- (p3_3_2_2);
  \draw[gray] (p4_1_1_1_1_1_1) -- (p2_2_1_1_1_1_1_1);
  \draw[gray] (p4_1_1_1_1_1_1) -- (p3_1_1_1_1_1_1_1);
  \draw[gray] (p4_2_1_1_1_1) -- (p2_2_2_1_1_1_1);
  \draw[gray] (p4_2_1_1_1_1) -- (p3_2_1_1_1_1_1);
  \draw[gray] (p4_2_1_1_1_1) -- (p4_1_1_1_1_1_1);
  \draw[gray] (p4_2_2_1_1) -- (p2_2_2_2_1_1);
  \draw[gray] (p4_2_2_1_1) -- (p3_2_2_1_1_1);
  \draw[gray] (p4_2_2_1_1) -- (p3_3_1_1_1_1);
  \draw[gray] (p4_2_2_1_1) -- (p4_2_1_1_1_1);
  \draw[gray] (p4_2_2_2) -- (p2_2_2_2_2);
  \draw[gray] (p4_2_2_2) -- (p3_2_2_2_1);
  \draw[gray] (p4_2_2_2) -- (p3_3_2_1_1);
  \draw[gray] (p4_2_2_2) -- (p4_2_2_1_1);
  \draw[gray] (p4_3_1_1_1) -- (p2_2_2_2_1_1);
  \draw[gray] (p4_3_1_1_1) -- (p3_2_2_1_1_1);
  \draw[gray] (p4_3_1_1_1) -- (p3_3_1_1_1_1);
  \draw[gray] (p4_3_1_1_1) -- (p4_2_1_1_1_1);
  \draw[gray] (p4_3_2_1) -- (p3_3_2_2);
  \draw[gray] (p4_3_2_1) -- (p4_2_2_1_1);
  \draw[gray] (p4_3_2_1) -- (p4_3_1_1_1);
  \draw[gray] (p4_3_3) -- (p3_3_3_1);
  \draw[gray] (p4_3_3) -- (p4_2_2_2);
  \draw[gray] (p4_3_3) -- (p4_3_2_1);
  \draw[gray] (p4_4_1_1) -- (p4_2_2_2);
  \draw[gray] (p4_4_1_1) -- (p4_3_2_1);
  \draw[gray] (p4_4_2) -- (p3_3_3_1);
  \draw[gray] (p4_4_2) -- (p4_4_1_1);
  \draw[gray] (p5_1_1_1_1_1) -- (p2_2_2_1_1_1_1);
  \draw[gray] (p5_1_1_1_1_1) -- (p3_2_1_1_1_1_1);
  \draw[gray] (p5_1_1_1_1_1) -- (p4_1_1_1_1_1_1);
  \draw[gray] (p5_2_1_1_1) -- (p2_2_2_2_1_1);
  \draw[gray] (p5_2_1_1_1) -- (p3_2_2_1_1_1);
  \draw[gray] (p5_2_1_1_1) -- (p3_3_1_1_1_1);
  \draw[gray] (p5_2_1_1_1) -- (p4_2_1_1_1_1);
  \draw[gray] (p5_2_1_1_1) -- (p5_1_1_1_1_1);
  \draw[gray] (p5_2_2_1) -- (p3_3_2_2);
  \draw[gray] (p5_2_2_1) -- (p4_2_2_2);
  \draw[gray] (p5_2_2_1) -- (p4_3_1_1_1);
  \draw[gray] (p5_2_2_1) -- (p5_2_1_1_1);
  \draw[gray] (p5_3_1_1) -- (p3_3_3_1);
  \draw[gray] (p5_3_1_1) -- (p4_2_2_2);
  \draw[gray] (p5_3_1_1) -- (p4_3_2_1);
  \draw[gray] (p5_3_1_1) -- (p5_2_1_1_1);
  \draw[gray] (p5_3_2) -- (p5_2_2_1);
  \draw[gray] (p5_3_2) -- (p5_3_1_1);
  \draw[gray] (p5_4_1) -- (p4_3_3);
  \draw[gray] (p5_4_1) -- (p4_4_2);
  \draw[gray] (p5_4_1) -- (p5_3_2);
  \draw[gray] (p5_5) -- (p5_4_1);
  \draw[gray] (p6_1_1_1_1) -- (p2_2_2_2_1_1);
  \draw[gray] (p6_1_1_1_1) -- (p3_2_2_1_1_1);
  \draw[gray] (p6_1_1_1_1) -- (p3_3_1_1_1_1);
  \draw[gray] (p6_1_1_1_1) -- (p4_2_1_1_1_1);
  \draw[gray] (p6_1_1_1_1) -- (p5_1_1_1_1_1);
  \draw[gray] (p6_2_1_1) -- (p5_2_2_1);
  \draw[gray] (p6_2_1_1) -- (p5_3_1_1);
  \draw[gray] (p6_2_1_1) -- (p6_1_1_1_1);
  \draw[gray] (p6_2_2) -- (p4_4_2);
  \draw[gray] (p6_2_2) -- (p5_3_2);
  \draw[gray] (p6_2_2) -- (p6_2_1_1);
  \draw[gray] (p6_3_1) -- (p4_3_3);
  \draw[gray] (p6_3_1) -- (p4_4_2);
  \draw[gray] (p6_3_1) -- (p5_3_2);
  \draw[gray] (p6_3_1) -- (p6_2_1_1);
  \draw[gray] (p6_4) -- (p5_4_1);
  \draw[gray] (p6_4) -- (p6_2_2);
  \draw[gray] (p6_4) -- (p6_3_1);
  \draw[gray] (p7_1_1_1) -- (p4_4_1_1);
  \draw[gray] (p7_1_1_1) -- (p6_2_1_1);
  \draw[gray] (p7_2_1) -- (p5_4_1);
  \draw[gray] (p7_2_1) -- (p6_2_2);
  \draw[gray] (p7_2_1) -- (p6_3_1);
  \draw[gray] (p7_2_1) -- (p7_1_1_1);
  \draw[gray] (p7_3) -- (p6_4);
  \draw[gray] (p7_3) -- (p7_2_1);
  \draw[gray] (p8_1_1) -- (p7_2_1);
  \draw[gray] (p8_2) -- (p5_5);
  \draw[gray] (p8_2) -- (p7_3);
  \draw[gray] (p8_2) -- (p8_1_1);
  \draw[gray] (p9_1) -- (p8_2);
  \draw[gray] (p10) -- (p9_1);
\end{tikzpicture}}
\caption{The order $\succeq$ of normalised immanants on $\TNN_{10}$.}\label{fig:hasse10}
\end{figure}

\clearpage

\section{Code and data}\label{app:data}

{\raggedright The code and data are in the repository \url{https://github.com/keithmcnulty/haiman-immanants-papers-n10-hooks}; the version used for this paper is the release tagged \texttt{v1.4.0}, archived as \cite{McNcode}. The computation of Section~\ref{sec:checks} and its results are unchanged since the first release, \texttt{v1.0.0} (\href{https://doi.org/10.5281/zenodo.23435088}{doi:10.5281/zenodo.23435088}); \texttt{v1.1.0} adds the checks of the rooted refinement \cite[Conj.~10.2]{McN}, \texttt{v1.2.0} the $n=10$ part of check~12, \texttt{v1.3.0} its test through all $42$ irreducible characters, and \texttt{v1.4.0} the check of the $g_T$ for $n=8$ and the present version of the text. Later releases make only documentation and packaging changes and change no result. All computations used Python with NumPy, Numba, SciPy and python-flint, together with \textsf{Coxeter3} \cite{Cox3} and SageMath for cross-checks. The repository contains:
\begin{itemize}
\item the extremal-pair engine \texttt{monimm/kl\_extremal.py} and its driver \texttt{run\_kl\_extremal.py}, with the result summary \texttt{results/ext\_n10.json};
\item the graded engine \texttt{monimm/kl\_graded.py} and its driver \texttt{run\_kl\_graded.py}, with results \texttt{results/kl\_graded\_n8.json}, \texttt{results/kl\_graded\_n9.json}, \texttt{results/kl\_graded\_n10.json} (cloud runbook \texttt{cloud/run\_kl\_graded\_n10.sh}), the induced-module test \texttt{tests/test\_kl\_graded.py}, and the chromatic check \texttt{check\_codominant\_chromatic.py} with \texttt{results/kl\_graded\_n10\_codominant.json};
\item the induced-module checks \texttt{check\_graded\_induced.py} ($n=6$, \texttt{results/graded\_induced\_n6.json}) and \texttt{check\_induced\_traces.py} with the depth-first kernel \texttt{monimm/induced\_traces.py} (\texttt{results/induced\_traces\_n8.json}, \texttt{results/induced\_traces\_n10.json}); at $n=10$ it reads the full graded table and the traces written by \texttt{cloud/run\_kl\_graded\_n10\_tables.sh} ($2.5$\,GB compressed), which are archived with the data record \cite{McNdata};
\item the seminormal-form check \texttt{check\_seminormal\_traces.py} with \texttt{monimm/seminormal\_traces.py} (\texttt{results/seminormal\_traces\_n8.json}, \texttt{results/seminormal\_traces\_n10.json}; cloud runbook \texttt{cloud/run\_seminormal\_n10.sh}, $29$ minutes on $192$ cores), which writes the table of all $\chi^\mu_q(T_v)$ mod $p$ at $q=2,3$, archived with the data record from version $2.1.0$ \cite{McNdata};
\item the exact inversion check \texttt{check\_inversion\_exact.py}, with results \texttt{results/inversion\_exact\_n9.json} and \texttt{results/inversion\_exact\_n10.json};
\item the \textsf{Coxeter3} comparison \texttt{cross\_check\_coxeter3.py}, with results \texttt{results/coxeter3\_check\_n*.json};
\item the SageMath comparison of check~9, \texttt{cross\_check\_sage.py}, with results \texttt{results/sage\_check.json} (every pair of $S_n\times S_n$ for $n\le5$, every extremal pair for $n=6$);
\item the full-interval engines \texttt{monimm/kl\_fast.py} and \texttt{monimm/kl\_big.py} ($n\le9$);
\item the order computations \texttt{monimm/ordering.py} and \texttt{monimm/order\_structure.py}, with results \texttt{results/immanant\_ordering\_n1-12.json} and \texttt{results/order\_structure\_n1-10.json};
\item the isolation study \texttt{monimm/tnn\_cone.py} and \texttt{check\_tnn\_isolation.py}, with results \texttt{results/tnn\_isolation\_n4.json} and \texttt{results/tnn\_isolation\_n5.json}, and the exact certificates for Theorem~\ref{thm:smalln}, \texttt{check\_isolation\_certificates.py} with \texttt{results/tnn\_isolation\_certificates\_n4.json} and \texttt{results/tnn\_isolation\_certificates\_n5.json} (exponent vectors, integer separating functionals, and exact non-vertex certificates; \texttt{--verify} re-checks a file in integer arithmetic);
\item the clean-step computations \texttt{notes/clean\_steps.py}, with results \texttt{results/clean\_steps\_n*.json}, the checks of \cite{McN} in \texttt{notes/hook\_proof/}, \texttt{notes/skandera\_chain.py}, \texttt{notes/cprod.py}, \texttt{notes/r2\_entry\_check.py}, \texttt{notes/pattern\_lifting.py} and, for the rooted refinement \cite[Conj.~10.2]{McN}, \texttt{notes/rooted\_vertex.py}, \texttt{notes/rooted\_graded.py} and \texttt{notes/rooted\_graded\_an.py} (logs \texttt{results/rooted\_*.log}), and the hook expansions and the counts for the shapes $(3,2^b,1^c)$ of Section~\ref{sec:evidence} in \texttt{check\_further\_computations.py}, with results \texttt{results/further\_computations\_n8.json};
\item a test suite covering every engine against brute force at small $n$.
\end{itemize}
The full tables for $n=10$ are archived separately as \cite{McNdata} (version 2.1.0, \href{https://doi.org/10.5281/zenodo.23532289}{doi:10.5281/zenodo.23532289}): the $42\times10!$ integers $c_{\lambda,w}$ (\texttt{ext\_n10.npz}), the $42\times10!$ polynomials $c_{\lambda,w}(q)$ (\texttt{kl\_graded\_n10.npz.zst}), the traces $\phi^\lambda_q(T_v)$ modulo $2^{31}-19$ at $q=2,3$ for every $v$ (\texttt{kl\_graded\_n10\_traces\_modp.npz}), and the irreducible characters $\chi^\mu_q(T_v)$ for every $\mu\vdash10$ and every $v$, with the same modulus and the same $q$ (\texttt{seminormal\_traces\_n10\_modp.npz}, new in version 2.1.0). Their SHA-256 checksums are in \texttt{results/SHA256SUMS}.

\emph{Reproducing the results.} Each result file has a driver and, for the large runs, a runbook that checks a smaller case first and must reproduce a stored result before the main run starts. The integer-safety bounds of Section~\ref{sec:checks} are recomputed and recorded by the drivers (\nolinkurl{accumulation_bound} and \nolinkurl{check_no_int64_overflow} in \nolinkurl{results/kl_graded_n10.json}). The command \texttt{python -m monimm.verify} runs the checks of Section~\ref{sec:checks} and prints one line per check, numbered as there: it reruns them for $n\le8$ (and $n=9$ with \texttt{--tier} \texttt{local}), checks the stored $n=10$ results against each other, against the numbers in this paper and against the checksums in \nolinkurl{results/SHA256SUMS}, and records the package versions; \nolinkurl{requirements-lock.txt} pins them.
\begin{center}
\footnotesize
\begin{tabular}{>{\raggedright\arraybackslash}p{4.3cm}>{\raggedright\arraybackslash}p{6.9cm}>{\raggedright\arraybackslash}p{2.6cm}}
\toprule
result & command (from the repository root) & resources\\
\midrule
\nolinkurl{results/ext_n10.json} & \texttt{bash} \nolinkurl{cloud/run_n10.sh} (runs \nolinkurl{run_kl_extremal.py} \texttt{--n} \texttt{10}; acceptance test $n=9$ against \nolinkurl{results/big_n9.npz}) & 279\,s, 192 cores, 15.5\,GB\\
\nolinkurl{results/kl_graded_n10.json} & \texttt{bash} \nolinkurl{cloud/run_kl_graded_n10.sh} (runs \nolinkurl{run_kl_graded.py} \texttt{--n} \texttt{10}; acceptance test $n=8$) & 1\,h\,31\,min, 192 cores, 125\,GB\\
\nolinkurl{results/kl_graded_n10_codominant.json} & \texttt{bash} \nolinkurl{cloud/run_kl_graded_n10_codominant.sh} & as above\\
\nolinkurl{results/coxeter3_check_n10.json} & \texttt{bash} \nolinkurl{cloud/run_coxeter3_n10.sh} (runs \nolinkurl{cross_check_coxeter3.py} \texttt{--n} \texttt{10} \texttt{--all}) & 2.4\,h, 190 cores\\
\nolinkurl{results/inversion_exact_n10.json} & \nolinkurl{check_inversion_exact.py} \texttt{--n} \texttt{10} \texttt{--coeffs} \nolinkurl{results/ext_n10.npz} & about 10\,GB\\
\nolinkurl{results/kl_graded_n9.json} & \nolinkurl{run_kl_graded.py} \texttt{--n} \texttt{9} & 25\,min, 4 cores\\
\nolinkurl{results/tnn_isolation_certificates_n5.json} & \nolinkurl{check_isolation_certificates.py} \texttt{--n} \texttt{5} & minutes, 1 core\\
test suite & \texttt{python} \texttt{-m} \texttt{pytest} \texttt{-q} & minutes\\
all checks, $n\le8$ & \texttt{python} \texttt{-m} \nolinkurl{monimm.verify} & 3\,min, 4 cores\\
\bottomrule
\end{tabular}
\end{center}
\par}

\section*{Declaration of generative AI and AI-assisted technologies in the manuscript preparation process}
During the preparation of this work the author used Claude (Anthropic), through Claude Code, to write and test the software, to run and cross-check the computations, to search and summarise the literature, to check arguments against computed examples, and to draft and edit the text. After using this tool, the author reviewed and edited the content as needed and takes full responsibility for the content of the published article.


\begin{thebibliography}{99}
\bibitem{AN} A.~Abreu and A.~Nigro, An update on Haiman's conjectures, Forum Math.\ Sigma 12 (2024), e86; arXiv:2206.00073v2.
\bibitem{CHSS} S.~Clearman, M.~Hyatt, B.~Shelton and M.~Skandera, Evaluations of Hecke algebra traces at Kazhdan--Lusztig basis elements, Electron.\ J.\ Combin.\ 23(2) (2016), P2.7.
\bibitem{CSS} S.~Clearman, B.~Shelton and M.~Skandera, Path tableaux and combinatorial interpretations of immanants for class functions on $S_n$, Discrete Math.\ Theor.\ Comput.\ Sci.\ Proc.\ AO (2011), 233--244.
\bibitem{Cox3} F.~du~Cloux, \textsf{Coxeter3}, software for Coxeter groups; maintained version at \url{https://github.com/tscrim/coxeter}.
\bibitem{DWG} X.~Dong, T.~Wu and X.~Gao, On the nonnegativity of monomial immanants for hook partitions, arXiv:2609.36665v1.
\bibitem{FG} C.~K.~Fan and R.~M.~Green, Monomials and Temperley--Lieb algebras, J.\ Algebra 190 (1997), 498--517.
\bibitem{FZ} S.~Fomin and A.~Zelevinsky, Double Bruhat cells and total positivity, J.\ Amer.\ Math.\ Soc.\ 12 (1999), 335--380.
\bibitem{GP} M.~Geck and G.~Pfeiffer, \emph{Characters of Finite Coxeter Groups and Iwahori--Hecke Algebras}, London Math.\ Soc.\ Monographs 21, Oxford University Press, 2000.
\bibitem{Hai} M.~Haiman, Hecke algebra characters and immanant conjectures, J.\ Amer.\ Math.\ Soc.\ 6 (1993), 569--595.
\bibitem{Hik} T.~Hikita, A proof of the Stanley--Stembridge conjecture, J.\ Amer.\ Math.\ Soc., published electronically 28 August 2026, \href{https://doi.org/10.1090/jams/1090}{doi:10.1090/jams/1090}; arXiv:2410.12758.
\bibitem{Hoe74} P.~N.~Hoefsmit, Representations of Hecke algebras of finite groups with BN-pairs of classical type, doctoral dissertation, University of British Columbia, 1974.
\bibitem{KL} D.~Kazhdan and G.~Lusztig, Representations of Coxeter groups and Hecke algebras, Invent.\ Math.\ 53 (1979), 165--184.
\bibitem{McN} K.~McNulty, Hook monomial immanants are Kazhdan--Lusztig positive, preprint, 2026, arXiv:2610.39021.
\bibitem{McNcode} K.~McNulty, \texttt{keithmcnulty/haiman-immanants-papers-n10-hooks}: code and data for the papers, version v1.4.0, software, Zenodo, 2026, \href{https://doi.org/10.5281/zenodo.23730750}{doi:10.5281/zenodo.23730750}.
\bibitem{McNdata} K.~McNulty, Kazhdan--Lusztig coefficients $c_{\lambda,w}$ and Haiman's polynomials $c_{\lambda,w}(q)$ of the monomial immanants for $S_{10}$, version 2.1.0, dataset, Zenodo, 2026, \href{https://doi.org/10.5281/zenodo.23532289}{doi:10.5281/zenodo.23532289}.
\bibitem{Pat91} T.~H.~Pate, Partitions, irreducible characters, and inequalities for generalized matrix functions, Trans.\ Amer.\ Math.\ Soc.\ 325 (1991), 875--894.
\bibitem{Ram} A.~Ram, A Frobenius formula for the characters of the Hecke algebras, Invent.\ Math.\ 106 (1991), 461--488.
\bibitem{Ram97} A.~Ram, Seminormal representations of Weyl groups and Iwahori--Hecke algebras, Proc.\ London Math.\ Soc.\ (3) 75 (1997), 99--133.
\bibitem{RS06} B.~Rhoades and M.~Skandera, Kazhdan--Lusztig immanants and products of matrix minors, J.\ Algebra 304 (2006), 793--811.
\bibitem{SW16} J.~Shareshian and M.~L.~Wachs, Chromatic quasisymmetric functions, Adv.\ Math.\ 295 (2016), 497--551.
\bibitem{Ska08} M.~Skandera, On the dual canonical and Kazhdan--Lusztig bases and $3412$-, $4231$-avoiding permutations, J.\ Pure Appl.\ Algebra 212 (2008), 1086--1104.
\bibitem{Ska25} M.~Skandera, Hook immanantal inequalities for totally nonnegative matrices, Electron.\ J.\ Combin.\ 33(4) (2026), P4.23, \href{https://doi.org/10.37236/14684}{doi:10.37236/14684}; arXiv:2510.00327v2.
\bibitem{Sot} F.~Sottile, Toric ideals, real toric varieties, and the moment map, in: Topics in Algebraic Geometry and Geometric Modeling, Contemp.\ Math.\ 334 (2003), 225--240.
\bibitem{Ste91} J.~Stembridge, Immanants of totally positive matrices are nonnegative, Bull.\ London Math.\ Soc.\ 23 (1991), 422--428.
\bibitem{Ste92} J.~Stembridge, Some conjectures for immanants, Canad.\ J.\ Math.\ 44 (1992), 1079--1099.
\bibitem{War} G.~S.~Warrington, Kazhdan--Lusztig polynomials of symmetric groups (data), \url{https://gswarrin.w3.uvm.edu/research/klc/klc.html}.
\end{thebibliography}
\end{document}